\documentclass[12pt]{article}

\usepackage{layout, graphicx, amsmath, amsthm, amssymb, euscript, enumerate, bbold, bbm, graphicx, fancyhdr,xcolor}

\usepackage[pdftex,breaklinks,pdfpagemode={None},pdfstartview={FitH},
pdfview={FitH},hypertexnames={false},pdfborder={0 0 0},colorlinks={false},
linkcolor={blue},citecolor={blue},urlcolor={blue}]{hyperref}

\newcommand{\footremember}[2]{
   \footnote{#2}
    \newcounter{#1}
    \setcounter{#1}{\value{footnote}}
}
\newcommand{\footrecall}[1]{
    \footnotemark[\value{#1}]
} 
\title{On branching process with rare neutral mutation\thanks{This research has been supported by CONACyT and CONACyT-CNRS Laboratorio Internacional Solomon Lefschetz.}}
\author{Airam Blancas Ben\'itez \footremember{cimat}{Centro de Investigaci\'on en Matem\'aticas (CIMAT A.C.). Calle Jalisco s/n, Col. Valenciana, 36240 Guanajuato, Gto. M\'exico. E-mail: airam@cimat.mx (Airam Blancas Ben\'itez), rivero@cimat.mx (V\'ictor Rivero).} \and V\'ictor Rivero\footrecall{cimat}}
\date{}

\newtheorem{theorem}{Theorem}[section]
\newtheorem{corollary}[theorem]{Corollary}

\newtheorem{lemma}[theorem]{Lemma}
\newtheorem{proposition}[theorem]{Proposition}
\makeatletter                                         
\def\th@newremark{\th@remark\thm@headfont{\bfseries}}  
\makeatletter                                          

\theoremstyle{newremark}                               
\newtheorem{remark}[theorem]{Remark}

\begin{document}

\numberwithin{equation}{section}

\maketitle

\begin{abstract}
In this paper we study the genealogical structure of a Galton-Watson process with neutral mutations, where the initial population is large and mutation rate is small \cite{B2}. Namely, we extend in two directions the results obtained in Bertoin's work. In the critical case, we construct the version of Bertoin's model conditioned not to be extinct, and in the case with finite variance we show convergence after normalization, of allelic sub-populations towards a tree indexed CSBP with immigration. 
Besides, we establish the version of the limit theorems in \cite{B2}, been for the unconditioned process and for the process conditioned to non-extinction, in the case where the reproduction law has infinite variance and it is in the domain of attraction of an $\alpha$-stable distribution. \end{abstract}
\smallskip
\noindent  \textbf{Keywords.} Branching process; Neutral mutations; Q-processes; Regular variation; Domain of attraction of $\alpha$-stable laws. 

\section{Introduction models and main results}
A Galton-Watson process models a population where at every generation each individual reproduces according to the same distribution, independently of the others and then dies. A number of variants, involving different types of conditioning and limit theorems, are core of branching processes theory. For instance, when the process dies with probability one, Yaglom (1947) proved that the distribution of the process conditioned to non-extinction exists, under some assumptions on the moments of the reproduction law. The proof was simplified and the moment assumptions removed by Joffe (1967),  Seneta and Vere-Jones (1967). More generally, Lamperti and Ney (1968) introduced the $Q$-process, conditioning on not being extinct in the distant future and being extinct in the even more distant future. For background about branching processes we refer to \cite{Athreya}, \cite{Amaurynotes} and \cite{Z.Li}.

As an extension of the Galton-Watson model, Bertoin \cite{B2} studied the so called Galton-Watson process with neutral mutations; this emerges assuming that the mutations modify the genotype of individuals but not the dynamic of a population modeled by a standard Galton-Watson. Since mutations appear in the ancestral lines of the population, each individual begets children that do not necessarily inherit its genetic type (allele). In addition, we suppose that the population has infinity alleles, that is, each mutation event originates a different allele. We denote the size of a typical family size by $\xi^{(+)} := \xi^{(c)}+\xi^{(m)}$, where $\xi^{(c)}$, $\xi^{(m)}$ are non-negative random variables which determine the number of clones and mutants children of a typical individual, respectively. We exclude the degenerate cases $\xi^{(c)}\equiv0$ or $\xi^{(m)}\equiv0$. 

In \cite{B2}, Bertoin established asymptotic features on the genealogy of allelic sub-families in a Galton-Watson process with neutral mutations. In his development, the genealogy of the population is described by a planar rooted tree where the mutations are represented by marks in the edges between parents and mutant childrens. The vertices with $n$ marks in their ancestral line are associated with the called $n$-type individuals; those individuals with parent of $(n-1)$-type are known mutants of the $n$-type. We denote by $T_n$ the total population of individuals of the $n$-th type and by $M_n$ the total number of mutants of $n$-th type, with the convention that mutants of the $0$-th type are the ancestors, that is $M_0=a$, $\mathbb{P}_a$-c.s. Observe that for every $n$ each mutant of the $n$-type begets a sub-family, which is independent of the others and has the same behavior that the original tree. In fact, this is a consequence of the general branching property, which states that conditionally on the set of children of a stopping line, the families that those beget are independent copies of the initial tree. The concept of stopping line was introduced formally by Chauvin \cite{Chauvin}. Namely, a line is a family of edges such that every branch from the root contains at most one edge in that family. One may think that a stopping line is a random line such that the event, ``an edge is in the line'', only depends on the marks found on their ancestral line. An important consequence from the general branching property is the following lemma.
\begin{lemma}[Bertoin \cite{B2}, Lemma 1] \label{L1B2}
Under $\mathbb{P}_a$ 
$$\{ M_n: n\in\mathbb{Z}_+  \},$$
is a Galton-Watson process with reproduction law $\mathbb{P}_1(M_1\in \cdot)$. More generally, $$\{ (T_n,M_{n+1}): n\in\mathbb{Z}_+  \},$$ is a Markov chain, with transition probabilities
\begin{equation}\label{probatranTM}
\mathbb{P}_a(T_n=k,M_{n+1}=l\left| \right. T_{n-1}=i,M_{n}=j ) = \mathbb{P}_j (T_0=k,M_1=l), \quad i,j, k, l\in\mathbb{Z}_+ \text{ and } j\leq k.
\end{equation}
\end{lemma}
\begin{remark}
Since the mutants of the $n$-th type are also individuals of the $n$-th type, the transition probability in (\ref{probatranTM}) is zero when $j > k$. 
\end{remark}
For notational convenience $\{ P^n_{(i,j),(k,l)} : i,j,k,l\in\mathbb{Z}_+  \}$ denotes the $n$-step transition probabilities of $\{ (T_n,M_{n+1}),n\in\mathbb{Z}_+ \}$, that is
\begin{equation} \label{nprobatranTM}
P^n_{(i,j),(k,l)}  = \mathbb{P}_a(T_{r+n}=k,M_{r+n+1}=l|T_{r}=i,M_{r+1}=j),  \quad i,j,k,l\in\mathbb{Z}_+, n\in\mathbb{N}. 
\end{equation}
\begin{remark}
For $ i,j,k,l\in\mathbb{Z}_+$,  $ P^n_{(i,j),(k,l)}$ depends only on the second coordinate, actually  is not difficult to prove using induction that 
\begin{equation}\label{probatranTMexp}
P^n_{(i,j),(k,l)} 
 = \sum_{j_{n-1}=1}^\infty \mathbf{P}^{n-1}_{(j,j_{n-1})} \mathbb{P}_{j_{n-1}}(T_0=k,M_1=l),
\end{equation}
where $j_0=j$ and $\{ \mathbf{P}^{n} _ {(i,j)} : i,j \in\mathbb{Z}_+ \}$ denotes the $n$-step transition probabilities of $\{ M_{n},n\in\mathbb{Z}_+ \}$. 
\end{remark}
Let $\mathbb{U}$ be the finite sequence of integers
$$\mathbb{U}:=\bigcup_{k\in\mathbb{Z}_+}\mathbb{N}^k,$$ 
where $\mathbb{N}=\{ 1,2,...\}$ and $\mathbb{N}^0= \{ \varnothing \}$. We consider a locally finite rooted tree $\mathcal{A}=\{ \mathcal{A}_u:u\in\mathbb{U} \}$ called tree of alleles. According to Ulam-Harris-Neveu labeling the root is $\varnothing$, a vertex at its level $k>0$ is $u=(u_1,...,u_k)$ and $uj=(u_1,...,u_k,j)$ is the $j$-th children. We denote by $|u|$ the level of the vertex $u$. A tree $\mathcal{A}$, is constructed recursively in \cite{B2} taking $\mathcal{A}_\varnothing=T_0$ and $\mathcal{A}_{uj}$ as the size of the $j$-th allelic sub-population of the type $|u|+1$ which descend from the allelic sub-family indexed by the vertex $u$, with the convention that in the case of ties, sub-families are ordered uniformly at random. Since the transition probabilities of the chain $\{ (T_n,M_{n+1}): n\in\mathbb{Z} \}$ depend only on the second coordinate, the tree of alleles enjoy a kind of branching property, the formal statement is given in the following result, where the notation $(d_u\downarrow)$ means that the $d_u$-tuple has been rearranged in the decreasing order of the first coordinate, by convention, in the case of ties, the coordinates are ranked uniformly at random. 

\begin{lemma}[Bertoin \cite{B2}, Lemma 2] \label{L2B2} 
For every integer $a\geq1$ and $k\geq 0$, under $\mathbb{P}_{a}$ conditionally on $( (\mathcal{A}_u,d_u): |u|\leq k)$, for each vertex $u$ at level $k$ with $\mathcal{A}_u>0$, the family of variables $( (\mathcal{A}_{uj},d_{uj}) : 1\leq j \leq d_u )$ are independent with distribution $(T_0,M_1)^{(d_u\downarrow)}$ under $\mathbb{P}_1$. 
\end{lemma}
\noindent It is important to observe that if $d_u:=\max\{ j\geq 1: \mathcal{A}_{uj}>0 \}$, agreeing that $\max \varnothing =0$, then
\begin{equation} \label{obsnivcsbptree}
T_k=\sum_{|u|=k} \mathcal{A}_u \quad \text{and} \quad M_{k+1}=\sum_{|u|=k} d_u.
\end{equation}
Hence given a population with neutral mutations, the tree of alleles records the genealogy of allelic sub-families together with their sizes. Also, the size of their generations is a Galton-Watson process. 

Besides, we consider for every integer $n$, a Galton-Watson process started from $n$ ancestors, with a fixed reproduction law, which is critical with finite variance and mutations occur at random with rate $1/n$. The main result of \cite{B2} states that as $n$ tends to infinity, the rescaled tree of alleles $n^{-2}\mathcal{A}^{(n)}$ converges in the sense of finite dimensional distributions towards a process $\{ \mathcal{Z}_u: u\in \mathbb{U} \}$, called tree-indexed CSBP with values in $(0,\infty)$, such that the vertices on a level are distributed as the family of the atoms of a Poisson random measure conditionally on the preceding  level. This result gives another point of view to the convergence of a Galton-Watson process rescaled towards a discrete time continuous state branching process, CSBP for short.

A first goal in this paper is to construct the chain $\{ (T_n,M_{n+1}): n\in\mathbb{Z}_+ \}$ conditioned on non-extinction of the mutants, so we are interested in the situation where
\begin{equation}\label{timext}
T=\inf\{ n\geq1 :M_n=0 \}<\infty, \quad \mathbb{P}_a\text{-}c.s.
\end{equation}
That is the purpose of the following Theorem. 
\begin{theorem}\label{Tpcondextdistfut}
Let $a\in\mathbb{Z}_+$ and $\mathcal{F}_n$ the natural filtration of the process $\{ (T_{n-1}, M_n) : n\in\mathbb{N} \}$. There exits a probability measure $\mathbb{P}^\uparrow_a$ that can be expressed as a $h$-transform of $\mathbb{P}_a$ using the $(\mathcal{F}_n)$-martingale
$$Y_n=\frac{ M_n q^{M_n-a} }{ ( f'(q) )^n }\mathbf{1}_{\{n<T\}},$$
where  
$f(y)=\mathbb{E}_1( y^{M_1} )$ and $q=\mathbb{P}(0<T<\infty)$.

That is
$$d\mathbb{P}^\uparrow_a |_{\mathcal{F}_n} = \frac{Y_n}{a}\,d\mathbb{P}_a |_{\mathcal{F}_n}, \quad n\in \mathbb{N}.$$
Furthermore, $\mathbb{P}^\uparrow_a$  is the law of a Markov chain $\{ (T^\uparrow_n,M^\uparrow_{n+1}),n\in\mathbb{Z}_+ \}$ with $n$-step transition probabilities, 
\begin{equation} \label{qtrans}
Q^n_{(i,j),(k,l)}= \frac{ l q^{l-j} }{ j( f'(q) )^n }P^n_{(i,j),(k,l)}, \quad j,l\geq 1,
\end{equation}
where $\{ P^n_{(i,j),(k,l)} : i,j,k,l\in\mathbb{Z}_+  \}$ denotes the $n$-step transition probabilities of $\{ (T_n,M_{n+1}),n\in\mathbb{Z}_+ \}$.
\end{theorem}
Next we ensure that the process defined in the above Theorem is distributed as $\{ (T_n,M_{n+1}): n\in\mathbb{Z} \}$ conditionally on non-extinction of mutants. Actually, thanks to a relation that we establish in Section \ref{pregf}, between the generating functions of $ (T_n,M_{n+1})$ and $M_{n}$, we can use classical methods and prove the following result. 
\begin{theorem}\label{Yaglomlimt}
Suppose that $\mathbb{E}(\xi^{(c)}) < 1$ and $\mathbb{E}(\xi^{(+)})\leq1$.
\begin{enumerate}
\item[i)]
The Yaglom limit 
\begin{equation*} \label{limyaglom}
\lim_{n\rightarrow\infty} \mathbb{P}(T_{n-1}=i,M_n=j|\,n<T<\infty),
\end{equation*}
exists and has a generating function $\widehat{\varphi}(x,y)$ such that for all $n\in\mathbb{N}$,
\begin{equation} \label{fglimyaglom}
m^n\widehat{\varphi}(x,y) = \widehat{f}(\varphi_n(x,y)) - \widehat{f}(\varphi_n(x,0)), \quad x,y\in[0,1]. 
\end{equation}
\item[ii)]
Let $a\in\mathbb{Z}_+$ and $n$ fixed. The conditional laws of the process $\{ (T_k,M_{k+1}): 0\leq k \leq n-1 \}$ under $\mathbb{P}_a(\cdot\,|n+k<T<\infty)$ converge towards the probability measure $\mathbb{P}_a^\uparrow$, in the sense that for any $n\geq0$ 
\begin{equation} \label{limdistfut}
\lim_{k\rightarrow \infty} \mathbb{P}_a(A|\,n+k<T<\infty) = \mathbb{P}_a^{\uparrow}(A), \quad \forall A\in\mathcal{F}_n.    
\end{equation}
\end{enumerate}
\end{theorem}
\begin{remark}\label{eqextimmean}
In the situation $\mathbb{E}(\xi^{(c)}) < 1$, according to Corollary 1 of \cite{B2}, the hypothesis $\mathbb{E}(\xi^{(+)})\leq1$ is equivalent to the condition (\ref{timext}). 
\end{remark}

Furthermore, in Section \ref{secinter} we will see that just as in the case of Galton-Watson processes conditioned to non-extinction, the dynamic of $\{ (T_n^\uparrow,M_{n+1}^\uparrow): n\in\mathbb{Z} \}$ admits a description in terms of immigration of mutants of the $n$-th type into the population. 

Besides we investigate, in the same sense of \cite{B2} but on a complimentary class of reproduction laws, the asymptotic behavior of the population. Namely, we consider a Galton-Watson process $\{Z^{+}_n,n\in\mathbb{Z}_+\}$ such that the reproduction law
\begin{equation*}
\pi^+_k=\mathbb{P}(\xi^{(+)}=k), \quad k\in\mathbb{Z}_+, 
\end{equation*}
is critical but, instead of assumming that it has finite variance as in Bertoin's paper, we suppose that there exists $\alpha\in(1,2)$ such that, 
\begin{equation} \label{hipcolapit}
\bar{\pi}^{+}(j):= \mathbb{P}( \xi^{(+)} >j )\in RV_\infty^{-\alpha},  
\end{equation}
where $RV_\infty^{-\alpha}$ denotes the class of functions which are regularly varying with index $-\alpha$ at $\infty$, see Chapter I in \cite{Teugles} for background. Note that the case where $\alpha\in(0,1)$ is not allowed here because it contradicts the assumption $\pi^+$ is critical. 

In order to extended the main result of \cite{B2} to our setting, we use that there exists a regularly varying function $r$ with index $1/\alpha$ such that
\begin{equation} \label{rpro}
r(n)\mathbb{P}( \xi^+ >ny ) \xrightarrow[n\rightarrow\infty]{} c_\alpha y^{-\alpha}, \quad \forall y>0,
\end{equation}
where $c_\alpha= 1/\Gamma(3-\alpha)$. A proof of this fact is given in Proposition \ref{convmedidas}. Also, we assume that each child is a clone of her mother with fixed probability $1-p$ and a mutant with probability $p$, so the joint law of $(\xi^{(c)}, \xi^{(m)})$, denoted by $\pi=(  \pi_{k,l} : k,l\in\mathbb{Z}_+ )$, 
\begin{equation} \label{lawreptypnot} 
\pi_{k,l} = \mathbb{P}( \xi^{(c)}=k, \xi^{(m)}=l ), \quad k,l\geq0,
\end{equation}
satisfies
\begin{equation}\label{mut}
\pi_{k,l}=\pi_{k+l}^{+}\binom{k+l}{k}(1-p)^k p^l, \quad k,l\in\mathbb{Z}_+.
\end{equation}
In the paper \cite{B2}, it has been assumed that the number of ancestors and mutation rate have the following behavior
\begin{equation} \label{hip}
a(n) \thicksim nx \quad \text{and} \quad p(n)\thicksim cn^{-1}, \quad \text{as } n\rightarrow\infty.
\end{equation} 
where $c,x$ are some positive constants. Here we will instead assume that following behaviour
\begin{equation} \label{hipnew}
a(n)  \thicksim x r(n)p(n)  \quad \text{and} \quad p(n)\thicksim cn^{-1}, \quad \text{as } n\rightarrow\infty.
\end{equation}
In the remainder of this paper the relation $f\thicksim g$ refers to $\lim_{x\rightarrow\infty}f(x)/g(x)=1$. The notation $\Longrightarrow$ refers to convergence in distribution as $n\rightarrow\infty$ and $\mathcal{L}\left( \cdot \,, \mathbb{P}_ {a(n)}^{p(n)} \right)$ to the distribution of the process under $\mathbb{P}_{a(n)}^{p(n)}$. The result below extends to our setting the main result in \cite{B2}.

\begin{theorem} \label{Tmainest}
If (\ref{hipcolapit}) and (\ref{hipnew}) holds. Then, the following convergence holds in the sense of finite dimensional distributions
$$ \mathcal{L}\left( ( ( (r(n) )^{-1}\mathcal{A}^{(n)}_u,(r(n)p(n))^{-1}d_u ^{(n)}):u\in\mathbb{U}),\mathbb{P}_{a(n)}^{p(n)} \right) \Longrightarrow ((\mathcal{Z}^{1/\alpha}_{u},\mathcal{Z}^{1/\alpha}_{u}):u\in\mathbb{U}),  $$
where $\{\mathcal{Z}_u:u\in\mathbb{U}\}$ is a tree-indexed CSBP with reproduction measure
\begin{equation} \label{nuest} 
\nu^\alpha(dy)=c'_\alpha x^{-1-1/\alpha}dy, \quad y>0, \,  \alpha\in(1,2),
\end{equation} 
where $c'_\alpha=\alpha^{-1}/ \Gamma(1 - \alpha^{-1} )$.
\end{theorem}

Finally we establish the convergence of the finite dimensional distributions of the rescaled chain $\{T_{n},M_{n+1} \}$, conditioned to non-extinction of mutants, towards a continuous state branching process with immigration in discrete time.
\begin{theorem} \label{T1}
If the reproduction law is critical, there exist sequences $b_1(n)$ and $b_2(n)$ such that the following joint convergence in the sense of finite dimensional distributions hold: 
\begin{equation*}
\mathcal{L}\left( ( b_1(n)T_{k-1}, b_2(n) M_{k} ) : k\in\mathbb{Z}_+ ) ,\mathbb{P}_{a(n)}^{p(n)\uparrow} \right)
\Longrightarrow 
\left( ( Y_k,\beta Y_k ) : k\in\mathbb{Z}_+ \right),
\end{equation*}
where $\{Y_k:k\in\mathbb{Z}_+\}$ is a CSBP with immigration such that
\begin{enumerate}
\item [a)] when the reproduction law has finite variance  $\sigma^2$ and (\ref{hip}) holds, then its reproduction measure is
\begin{equation} \label{fdpnu}
\nu(dy)=\frac{c}{\sqrt{2\pi\sigma^2y^3}}\exp\left(-\frac{c^2y}{2\sigma^2}\right)dy, \quad y>0, 
\end{equation}
then the immigration measure is $z\nu(dz)$ and $\beta=c$; moreover $b_1(n)= n^{-2}$ and $b_2(n)=n^{-1}$;
\item [b)] when the assumptions (\ref{hipcolapit}) and (\ref{hipnew}) hold, the reproduction measure is $\nu^\alpha(dz)$ as defined in (\ref{nuest}), the immigration measure is $z\nu^\alpha(dz)$ and $\beta=1$; the normalizing constants are given by $b_1(n)= (r(n)p(n))^{-1}$ and $b_2(n)=(r(n))^{-1}$.
\end{enumerate}
\end{theorem}

The remainder of the paper is organized as follows. In Section \ref{pregf} we provide some preliminary facts. Section \ref{seccond}, is devoted to construct and to interpret the process $\{ (T_n,M_{n+1}): n\in\mathbb{Z} \}$ conditioned to non extinction. The last section is divided in three parts; the first one correspond to the framework where the reproduction law is in the domain of attraction of an $\alpha$-stable distribution. Section \ref{sectpruebas} includes the proof of our main result in this setting. Finally we prove Theorem \ref{T1}, which is the result that explains the asymptotic behaviour of the process conditioned to non-extinction. 

\section{Preliminaries}\label{pregf}
In this section we obtain some useful formulas for the generating function of $(T_n,M_{n+1})$, denoted for $n\in\mathbb{Z}$ by 
\begin{equation*}
\varphi_n(x,y):=\mathbb{E}_1( x^{T_{n-1}} y^{M_n} ), \quad x,y\in [0,1],
\end{equation*}
where for notational convenience $\varphi_1(x,y):=\varphi(x,y)$. Observe that the generating function of $M_n$ is
\begin{equation} \label{relgenMgenMT}
f_n(y):=\varphi_n(1,y), \quad y\in [0,1],
\end{equation}
and as before we denote $f_1(y)=:f(y)$. 

According to classical theory of branching processes, the extinction probability of the Galton-Watson process $\{ M_{n},n\in\mathbb{Z}_+ \}$, that we denote by $q$, is the smallest root of $f(y)=y$, which is less or equal than one depending on whether, the mean of the reproduction law, $m:=\mathbb{E}(M_1)$ is $>1$, or $\leq1$, respectively. 
In order to avoid trivial cases, we assume throughout that 
\begin{enumerate}
\item [H1)]$\mathbb{P}(M_1=1)>0$,
\item [H2)]$\mathbb{P}(M_1=0)+\mathbb{P}(M_1=1)<1$, and
           $\mathbb{P}(M_1=j)\neq1$, for any $j$.
\end{enumerate}
We also know that the $n$-step transition probabilities $\{ \mathbf{P}^{n} _ {(i,j)} : i,j \in\mathbb{Z}_+ \}$ of $\{ M_{n},n\in\mathbb{Z}_+ \}$ satisfy
\begin{equation} \label{ptgenM} 
\sum_{j=0}^\infty \mathbf{P}^n_{(i,j)} y^j = \left( f_n(y) \right)^i, \quad i\geq1.
\end{equation}
Now for a Galton-Watson process with neutral mutations, let $g$ be the generating function of the reproduction law of a typical individual, that is
$$ g(x,y):= \mathbb{E}(x^{\xi^{(c)}} y^{ \xi^{(m)} } ), \quad x,y\in[0,1]. $$
Proposition 1 of \cite{B2} ensure that the law of $(T_0,M_1)$ can be obtained applying the Lagrange inversion formula to the equation
\begin{equation} \label{P1iB2}
\varphi(x,y)= xg(\varphi(x,y),y), \quad x,y\in [0,1],
\end{equation}
which gives the identity
\begin{equation} \label{P1iiB2}
\mathbb{P}_a( T_0=k, M_1=l ) = \frac{a}{k} \pi^{*k}_{k-a,l}, \quad k\geq a\geq 1 \text{ and } l\geq 0,
\end{equation}
where $\pi^{*k}$ denotes the $k$-th convolution of $\pi$, as defined in (\ref{lawreptypnot}).

The statement in the remark \ref{eqextimmean} is a consequence of (\ref{P1iB2}). We also have, $\mathbb{E}(\xi^{(+)2})<\infty$ if and only if $\mathbb{E}(M_1^2)<\infty$. From (\ref{P1iiB2}) we can write the hypothesis (H1) and (H2) in terms of the reproduction distribution of a typical individual. Besides, we have a similar identity to (\ref{ptgenM}) for the process $\{ (T_n,M_{n+1}),n\in\mathbb{Z}_+ \}$ 
\begin{equation} \label{ptgen}
\sum_{k,l=0}^\infty P^n_{(i,j),(k,l)} x^k y^l = \left( \varphi_n(x,y) \right)^j, \quad i,j\geq 1,
\end{equation}
where $P^n$ denotes the $n$-step transition probabilities.
We get this equality by induction, for we apply the Chapman-Kolmogorov equation to express the $(n+1)$-step transition probabilities in terms of the transitions in one step and use $(\ref{probatranTM})$.
A simple but key relation for our analysis is
\begin{equation} \label{ecgh}
\varphi_n(x,y) = f_{n-1}( \varphi(x,y) ), \quad x,y\in [0,1].
\end{equation}
Due to (\ref{relgenMgenMT}), the proof of the identity is equivalent to establish the following equality
\begin{equation}
\varphi_n(x,y)= \varphi_{n-1} (1,\varphi(x,y)), \quad x,y\in [0,1] , \label{genrec}
\end{equation}
which holds from the following standard calculations
\begin{align*} 
\varphi_n(x,y) 
& = \mathbb{E}_1( \mathbb{E}_1 ( x^{T_{n-1}} y^{M_n} \left |  \right. T_{n-2} , M_{n-1}  ) ) \\ 
& = \sum_{i,j=0}^\infty \mathbb{P}_1 ( T_{n-2}=i, M_{n-1}=j ) 
                                 \sum_{k=j}^\infty \sum_{l=0}^\infty x^k y^l \mathbb{P}_1(T_{n-1}=k,M_n=l \left| \right. T_{n-2}=i, M_{n-1}=j ) \\ 
& = \sum_{i,j=0}^\infty \mathbb{P}_1 ( T_{n-2}=i, M_{n-1}=j )  \sum_{k=j}^\infty \sum_{l=0}^\infty x^k y^l \mathbb{P}_j(T_0=k,M_1=l) \\ 
         & = \sum_{i,j=0}^\infty \mathbb{P}_1 ( T_{n-2}=i, M_{n-1}=j ) (\varphi(x,y))^j \label{genrec2}  \\  
         & = \varphi_{n-1} (1,\varphi(x,y)),
\end{align*}
where we use that $\{ (T_{n-1}, M_n) : n\in\mathbb{N} \}$ is Markovian, Lemma \ref{L1B2} and the branching property. 

\section{The process conditioned to non-extinction} \label{seccond}
This section is devoted to study the process $\{ (T_n,M_{n+1}),n\in\mathbb{Z}_+ \}$ conditioned to non-extinction. 

\subsection{Construction}
Here our aim is to prove Theorem \ref{Tpcondextdistfut}, that ensures the existence of a Markovian process that we shall understand as the chain $\{ (T_n,M_{n+1}),n\in\mathbb{Z}_+ \}$, conditioned to the non-extinction of mutants in the population.
\begin{proof}[Proof of Theorem \ref{Tpcondextdistfut}]
An application of the monotone convergence Theorem along with an elementary computation shows that
$$
\dfrac{d}{d s}\mathbb{E}_a ( s^{M_n} ) \left | \right._{s=q} =\mathbb{E}_a \left(  M_n q^{ M_n-1} \right).
$$ 
Moreover, the following identity is deduced form the branching property of the Galton-Watson process $\{ M_n: n\in\mathbb{Z}_+\}$ and the properties of its generating function 
$$\dfrac{d}{d s}\mathbb{E}_a ( s^{M_n} ) \left | \right._{s=q} =aq^{a-1} f_n'(q).$$
The latter and former identities imply in turn that
$$\mathbb{E}_a \left(  M_n q^{ M_n-1} \right) =a q^{a-1} f'_n(q).$$
Then by the Markov property, 
$$\mathbb{E}_a \left(  M_{n+k} q^{ M_{n+k}-1} | \mathcal{F}_n \right)
=M_n q^{M_n-1} f'_k(q).$$
Combining the latter with the fact that $f'_k(q)=\left[ f'(q) \right]^k$ (see Athreya and Ney \cite{Athreya}, Lemma 3.3), we have that $$Y_n=\frac{ M_n q^{M_n-a} }{ ( f'(q) )^n },$$ is martingale. Now from the theory of $h$-transform there exists a Markovian process (see \cite{chungwalsh}), that we denote by $\{ (T^\uparrow_n,M^\uparrow_{n+1}),n\in\mathbb{Z}_+ \}$ whose law satisfies
\begin{equation}  \label{pflecha}
\mathbb{P}_a^\uparrow(T^\uparrow_0=i_0,M^\uparrow_1=j_1, \cdots,T^\uparrow_{n-1}=i_{n-1},M^\uparrow_n=j_n): = \mathbb{P}_a(A_n) \frac{ j_n q^{j_n-a} }{ a( f'(q) )^n },
\end{equation}
where for every $n\in\mathbb{N}$, and
\begin{equation} \label{eventoAn}
A_n=\{ T_0=i_0,M_1=j_1, \cdots,T_{n-1}=i_{n-1},M_n=j_n \}, \quad i_0,j_1,...,i_{n-1},j_n\in\mathbb{N}.
\end{equation}
\end{proof}

\subsection{Conditional laws}\label{sec2}
This subsection is devoted to prove Theorem \ref{Yaglomlimt}, in this aim the keystone will be the generating function of $\{ (T_n,M_{n+1}): n\in\mathbb{Z} \}$, hence some of the results given in section \ref{pregf} will be neccesary. 

\begin{proof}[Proof of Theorem \ref{Yaglomlimt}]\label{pTlimtcond}
\begin{enumerate}
\item[i)]
We will first ensure that the generating function converges. For this end observe that $\{M_n>0\}$ under the event $\{ n<T<\infty \}$. Then for all $x,y\in[0,1]$, we have the identity
\begin{align*}
\widehat{\varphi}_n(x,y) 
& := \mathbb{E}_1( x^{T_{n-1}} y^{M_n} \left| \right. n<T<\infty ) \nonumber \\
&  = \frac{ \varphi_n(x,y) - \varphi_n(x,0) }{ 1-\mathbb{P}(M_n=0) }. \nonumber 
\end{align*}
Using the identity (\ref{ecgh}) and the fact $f_n(0)=\mathbb{P}(M_n=0)$, we get the expression.
$$ 
\widehat{\varphi}_n(x,y) 
=    \frac{ 1-f_{n-1}(0) }{ 1-f_n(0) } \left( \frac{ f_{n-1}(\varphi(x,y))- f_{n-1}(0) }{ 1-f_{n-1}(0) }
                       - \frac{ f_{n-1}(\varphi(x,0))- f_{n-1}(0) }{ 1-f_{n-1}(0) } \right). \nonumber 
$$
Now we take $u=f_{n-1}(0)$ and use that $m=f'(1)$ to obtain the following limit.
\begin{equation*}
\lim_{n\to\infty} \frac{ 1-f_{n-1}(0) }{ 1-f_n(0) } =  \lim_{u\to1} \frac{ 1-u }{ 1-f(u) } =  \frac{1}{m}.
\end{equation*}
Besides, note that for each $s$ the function 
$$n \longmapsto \frac{ 1-f_{n-1}(s) }{ 1-f_{n-1}(0) },$$
is a decreasing function. This implies that as $n$ tends to infinity the expression
\begin{equation*}
 \frac{ f_{n-1}(s)- f_{n-1}(0) }{ 1-f_{n-1}(0) } = 1 - \frac{ 1-f_{n-1}(s) }{ 1-f_{n-1}(0) },
\end{equation*}
has a limit, say $1 - \widehat{f}(s)$. We stress (see for instance Athreya and Ney \cite{Athreya}, Theorem 8.1) that $ \widehat{f}(s)$ is the generating function of  $\widehat{m}=\{ \widehat{m}_k : k\in\mathbb{N} \}$, the Yaglom distribution of $\{M_n: n\in\mathbb{Z}_+\}$, which is obtained for all $k\in\mathbb{N}$, as follows
$$
 \widehat{m}_k = \lim_{n\rightarrow\infty} \mathbb{P}(M_n=k|\,n<T<\infty).
$$
This fact implies that, the generating function $\widehat{f}$ satisfies
\begin{equation} \label{fgqsdmut}
1-\widehat{f}(f(s)) = m(1-\widehat{f}(s)), \quad s\in[0,1].
\end{equation}
The above calculations imply the following 
\begin{equation*} 
\widehat{\varphi}(x,y)  := \lim_{n\rightarrow\infty} \widehat{\varphi}_n(x,y) =\frac{\widehat{f}(\varphi(x,y)) - \widehat{f}(\varphi(x,0))}{m} .
\end{equation*}
Now we prove by induction (\ref{fglimyaglom}). If $n=1$, it is the just proved equality. Then suppose  (\ref{fglimyaglom}) holds for $n=k$. In order to get the identity for $n=k+1$ note that by the induction hypothesis$$ 
m^{k+1}\widehat{\varphi}(x,y) = m \left[1-\widehat{f}(\varphi_k(x,0)) \right] - m \left[1-\widehat{f}(\varphi_k(x,y))\right], \quad x,y\in[0,1]. 
$$
From where we deduce the claim by using first (\ref{fgqsdmut}) and then (\ref{ecgh}).
\item[ii)]
Let us consider the event of the form in (\ref{eventoAn}). Since $\{ (T_n,M_{n+1}),n\in\mathbb{Z}_+ \}$ is Markovian we have for $A_n$ as in (\ref{eventoAn}), that
$$
\mathbb{P}_a(A_n,n+k<T<\infty) =  \mathbb{E}_a( \mathbf{1}_{A_n} \mathbb{E}_{j_n}( \mathbf{1}_{\{ M_k>0 \}} q^{M_k} ) ),
$$
and similarly
$$ 
\mathbb{P}_a(n+k<T<\infty)  = \mathbb{E}_a( \mathbf{1}_{\{ M_{n+k}>0 \}} q^{M_{n+k}}  ).
$$
Then using (\ref{ptgenM}), we get 
$$\mathbb{P}_a(A_n\,|n+k<T<\infty) 
= \mathbb{P}_a(A_n) \frac{\sum\limits_{j=1}^\infty \mathbf{P}^k_{(j_n,j)}q^j}{\sum\limits_{j=1}^\infty\mathbf{P}^{n+k}_{(1,j)}q^j}. 
$$
Theorem 7.4 in \cite{Athreya} establishes that the following limit holds  
$$
\lim_{k\rightarrow \infty }\frac{ \mathbf{P}^{n+k}_{(i_1,j)} }{ \mathbf{P}^{k}_{(i_2,j)} } = i_1 i_2 ^{-1} ( f'(q) )^k   q^{i_1-i_2}.
$$
Finally, thanks to hypothesis (H2) we can apply the previous result to obtain 
$$
\lim_{k\rightarrow \infty} \mathbb{P}_a(A_n|n+k<T<\infty) = \mathbb{P}_a(A_n) \frac{ j_n q^{j_n-a} }{ a( f'(q) )^n },   \quad a\in\mathbb{N}, 
$$
which finishes the proof.
\end{enumerate}
\end{proof}
\begin{remark}
In the previous proof we established the existence of a Yaglom limit when $m\leq 1$, however similar arguments can be used to show that existences of a Yaglom limit in the supercritical case.
\end{remark}

\subsection{Interpretation}\label{secinter}
Motivated by the interpretation of a Galton-Watson process conditioned to non-extinction given in \cite{Amaury}, in the present subsection we provide a description of $\{ (T_n^\uparrow,M_{n+1}^\uparrow): n\in\mathbb{Z} \}$ as immigration of mutants of the $n$-th type into the population. We start calculating the generating function of the $n$-step transition probabilities of the process. 
\begin{proposition} 
The generating function of the $n$-step transition probabilities for the process $$\{ (T^\uparrow_n,M^\uparrow_{n+1}),n\in\mathbb{Z}_+ \}$$ is given by
\begin{equation} \label{qinmTM}
\sum_{k,l=1}^\infty Q^n_{(i,j),(k,l)} x^k y^l 
= \frac{yq^{1-j}}{ [ f'(q) ]^n}  [\varphi_n(x,qy)]^{j-1} \dfrac{\partial}{\partial y} \varphi(x,qy) \prod_{i=1}^{n-1}  f'( \varphi_i(x,qy)), \quad x,y\in[0,1].
\end{equation}
\end{proposition}

\begin{proof}
Since $x,y\leq 1$ the generating function of (\ref{qtrans}) is infinitely differentiable, then an elementary calculation and the use of (\ref{ptgen}) imply the following identity.
\begin{equation*}
\sum_{k,l=1}^\infty Q^n_{(i,j),(k,l)} x^k y^l 
= \frac{yq^{1-j}}{[f'(q)]^n} \left[  \varphi^{j-1}_n(x,u) \dfrac{\partial}{\partial u} \varphi_n(x,u) \right] _{u=qy}.
\end{equation*}
To finish the proof we apply repeatedly (\ref{ecgh}) and the recursion $f_n(y)=f(f_{n-1}(y))$.
\end{proof}
Taking $x=1$ in (\ref{qinmTM}) and recalling the fact that the transition probabilities of $\{ (T_n,M_{n+1}),n\in\mathbb{Z}_+ \}$ depend only on the second coordinate, we can identify a Galton-Watson process with immigration \cite{Watanabe}. 
 
 \begin{corollary} \label{corodbi}
If $\{ M_n,n\in\mathbb{Z}_+ \}$ is critical or subcritical, then $\{ M_n^\uparrow-1 ,n\in\mathbb{Z}_+ \}$ is a Galton-Watson process with immigration $[f,\frac{f'}{m}]$.
\end{corollary}

\noindent Note that $\{ M_n^\uparrow,n\in\mathbb{Z}_+ \}$ is the $Q$-process associated to the Galton-Watson process $\{ M_n,n\in\mathbb{Z}_+ \}$ (see for instance \cite{Amaury} or \cite{Athreya}). The following Corollary is analogous to Proposition 1 in \cite{B2}.

\begin{corollary} \label{coroextBP1}
If $\{ M_n,n\in\mathbb{Z}_+ \}$ is critical or subcritical, then the generating function of $(T_0^\uparrow,M_1^\uparrow)$ is given by the equation
$$ \mathbb{E}_1(x^{T^\uparrow_0}y^{M^\uparrow_1}) = \frac{xy}{m} \dfrac{\partial}{\partial y} g( \varphi(x,y),y ), \quad x,y\in[0,1]. $$
Moreover, the distribution of $(T_0^\uparrow,M_1^\uparrow)$ is given by
$$ \mathbb{P}_a^\uparrow( T_0^\uparrow=k, M_1^\uparrow=l ) = \frac{l}{mk} \pi^{*k}_{k-a,l}, \quad k\geq a\geq 1 \text{ y } l\geq 0. $$
\end{corollary}
\begin{proof}
Taking $n=1$ in the equality (\ref{qinmTM}),
$$\mathbb{E}_1(x^{T^\uparrow_0}y^{M^\uparrow_1})= \frac{y}{m}  \dfrac{\partial}{\partial y} \varphi(x,y). $$
Then the first identity is obtained with the substitution of (\ref{P1iB2}). To get the second one, recall the definition of $\mathbb{P}^\uparrow$ given in (\ref{pflecha}), then use (\ref{P1iiB2}). 
\end{proof}

We can now give an interpretation to the process $\{ (T^\uparrow_n,M^\uparrow_{n+1}),n\in\mathbb{Z}_+ \}$, in terms of a tree of alleles with immigration $\mathcal{A}^\uparrow=(\mathcal{A}^\uparrow_u:u\in\mathbb{U})$, this will provide a description of the genealogical structure in a population conditioned to non extinction. The key elements are the tree of alleles $\mathcal{A}=(\mathcal{A}_u:u\in\mathbb{U})$ and the above corollaries. 

We start defining $\mathcal{A}^\uparrow_\varnothing=T_0^\uparrow$ that is, the total number of individuals without mutations into the population, then according to a distribution with generating function $f'/m$, a random number of individuals of the same genetic type arrive. We enumerate the $M^\uparrow_1$ allelic sub-populations of the first type beget by $T_0^\uparrow$ in decreasing order, with the convention that in the case of ties, sub-populations of the same size are ranked uniformly at random.
Then, using Corollary \ref{corodbi} we choose uniformly at random one of the first type sub-families in the tree of alleles, removing it and replace it by a population of size $T_0^\uparrow$ which begets allelic subpopulation according to $M_1^\uparrow$, where $(T_0^\uparrow,M_1^\uparrow)$ is given by Corollary \ref{coroextBP1}. We continue with the construction by iteration, $\mathcal{A}^\uparrow_{uj}$ is the size of the $j$-th sub-population allelic of type $|u|+1$ which descend from the allelic sub-family indexed by the vertex $u$, then we choose one of the sub-families of the type $|u|+1$ to replace it for one of size $T_0^\uparrow$ which begets allelic subpopulation according to $M_1^\uparrow$. 

\section{Asymptotic behavior}\label{secab}
 
\subsection{The $\alpha$-stable case}
Consider a Galton-Watson process $\{Z^{+}_n,n\in\mathbb{Z}_+\}$ with reproduction law $\xi^+$ denoted by $\pi^{+}$, that is 
$$
\pi^+_k=\mathbb{P}(\xi^{(c)}+\xi^{(m)}=k), \quad k\in\mathbb{Z}_+.
$$
Suppose that $\pi^+$ is critical and (\ref{hipcolapit}) holds. The mutations appear in the population according to (\ref{mut}) and we focus on the situation (\ref{hipnew}) described in the Introduction. Our goal in this section is to prove Theorem \ref{Tmainest}, for that end we start by describing the normalizing constant $r(n)$ appearing therein.
\begin{lemma} \label{convmedidas}
If condition (\ref{hipcolapit}) holds, then there exists $r(n)\in RV_\infty^{1/ \alpha}$ such that
\begin{equation*}
r(n)\pi^{+}(ndy) \xrightarrow[n\rightarrow \infty] {} 
c_\alpha \frac{dy }{y^{1+\alpha}}, 
\end{equation*}
in the sense of vague convergence on $(0,\infty)$, where $c_\alpha= 1/\Gamma( 3 - \alpha)$. In particular
\begin{equation*}
\exp \left\{ -t \int_{[0,\infty)} (1-e^{-\lambda y}- \lambda y ) r(n)\pi^{+}(ndy) \right\}\xrightarrow[n\rightarrow \infty] {} 
e^{-t\lambda^{\alpha}}. 
\end{equation*}
\end{lemma} 
\noindent The proof is an elementary application of standard results from the theory of Regular Variation (see e.g. \cite{Teugles} for background), but we include a proof in Appendix \ref{prn} for sake of completeness.

\subsubsection{An approximation for the reproduction law}
In order to link the asymptotic behaviour of the reproduction law of a typical individual with that of the joint distribution of clones and mutants, we first link their Laplace transform. Although in the present setting we use some ideas of the standard Tauberian-Abelian Theorem, we remark that it is not straightforward application of this theorem because in the present setting we consider sequences of measures indexed by the positive integers that change, unlike to the standard case where only the normalizing constants change.
\begin{lemma} \label{TLlrlclm} 
For every positive integer $n$, let $\phi_n$ be the Laplace transform of $\xi^{(n)}=(\xi^{(cn)},\xi^{(mn)})$ under the measure $\mathbb{P}^{p(n)}$. Assume $\{\lambda(n): n\geq 0 \}$ is a positive sequence such that $\lambda(n)\rightarrow 0$, as $n\rightarrow \infty$. Then
\begin{equation} \label{TLlrlclmec} 
\phi_n(\lambda(n),\theta) \thicksim \phi^{+} \left( (1-p(n))(1-e^{-\lambda(n)}) + p(n)(1-e^{-\theta} ) \right) \quad  \text{ with } n\rightarrow\infty,  \forall  \theta\geq0,
\end{equation}
where $\phi^{+}$ is the Laplace transform of $\xi^{(+)}$.
In particular, $\phi^{\,m}_n$, respectively $\phi^{\,c}_n$, the Laplace transform of $\xi^{(mn)}$, respectively $\xi^{(cn)}$, satisfy
\begin{align*}
\phi_n^{\,m}(\theta)  & \thicksim  \phi^{+} ( \,p(n) (1-e^{-\theta} ) \, ), \\
[\phi_n^{\,c}(\lambda(n))  & \thicksim  \phi^{+} ( \,(1-p(n)) (1-e^{-\lambda(n)}) \, )],  \quad  \text{ as } n\rightarrow\infty, \forall  \theta\geq0 .
\end{align*}
\end{lemma}
\begin{proof}
According to (\ref{mut}), conditionally to $\xi^{(+)}=k$ the distribution of $\xi^{(m)}$ is Binomial with parameter $(k,p)$. This fact implies the following equality in law  
\begin{equation} \label{xicmunif}
(\xi^{(c)},\xi^{(m)}) \overset {\mathcal{L}} { = }\sum_{i=1}^{\xi^{(+)}}( \mathbf{1}_{\{ U_i > \, p \}}, \mathbf{1}_{\{ U_i \leq \, p \}} ), 
\end{equation}
where $\{U_i, i\geq 1\}$ are i.i.d. random variables with common distribution that of an uniform random variable in $(0,1)$.
Therefore,
\begin{align*}
\phi_n(\lambda(n),\theta)
& = \sum_{k=0}^\infty \mathbb{P}(\xi^{(+)}=k) \left[ (1-p(n))e^{-\lambda(n)} + p(n)e^{-\theta} \right] ^k \\
& = \phi^{+} (  -\log ( 1 - ( 1- (1-p(n))e^{-\lambda(n) } - p(n)e^{-\theta} ) ) ). 
\end{align*}
We conclude the proof using (\ref{hip}) and the elementary asymptotic estimate
\begin{equation} \label{limlog}
\frac{\log(1-y)}{y} \xrightarrow[y\rightarrow 0]{}  -1.
\end{equation}
\end{proof}
In the same way it is posible to establish the following estimate.
\begin{corollary} \label{corTLlrlclm} 
For every positive integer $n$, let $\psi_n$ be the characteristic function of $\xi^{(n)}=(\xi^{(cn)},\xi^{(mn)})$ under the measure $\mathbb{P}^{p(n)}$. Then
\begin{equation} \label{TLlrlclmec2} 
\psi_n(\lambda(n),\theta)   \thicksim  \phi^{+} ( \, (1-p(n))(1-e^{\text{i}\lambda(n)}) + p(n)(1-e^{\text{i}\theta} )  \, ),  \quad  \text{ as } \lambda(n)\xrightarrow [n\rightarrow\infty]{}0,  \forall  \theta\geq0.
\end{equation}
In particular $\psi^{\,m}_n$, respectively $\psi^{\,c}_n$, the characteristic function of $\xi^{(mn)}$ respectively of $\xi^{(cn)}$, satisfies
\begin{align*}
\psi_n^{\,m}(\theta)  & \thicksim  \phi^{+} ( \,p(n) (1-e^{i\theta} ) \, ), \\
[\psi_n^{\,c}(\lambda(n))  & \thicksim  \phi^{+} ( \,(1-p(n))(1-e^{\text{i}\lambda(n)}) \, )], 
\end{align*}
as $n\rightarrow\infty$, $\lambda(n)\rightarrow 0$ and for all $\theta\geq0$ .
\end{corollary}
\begin{proof}
Similarly to previous Lemma, using (\ref{xicmunif}) we have 
\begin{align*}
\psi_n(\lambda(n),\theta)
& = \sum_{k=0}^\infty \mathbb{P}(\xi^{(+)}=k) \left[ (1-p(n))e^{\text{i}\lambda(n)} + p(n)e^{\text{i}\theta} \right] ^k \\
& = \sum_{k=0}^\infty \mathbb{P}(\xi^{(+)}=k) \exp\{ k \log ( 1 - ( (1-p(n))(1-e^{\text{i}\lambda(n) }) - p(n)( 1- e^{\text{i}\theta} ) )) \} 
\end{align*}
To conclude we apply the asymptotic estimate (\ref{limlog}).
\end{proof}
We can now use the results above to give an estimation for the reproduction measure. 
\begin{proposition} \label{colaaproxcmn} 
For every positive integer $n$, let $\pi^{(cn)}$, $\pi^{(mn)}$ be the reproduction laws of $\xi^{(cn)}$ and $\xi^{(mn)}$, respectively. Assume $\{y(n): n\geq 0\}$ is any sequence such that $y(n)\rightarrow\infty$ as $n\rightarrow\infty$. In the regime (\ref{hipcolapit}), the asymptotic behaviour of the tail distribution of $\xi^{(\cdot)}$ is given by
\begin{equation} \label{colaaproxcmneq}
\bar{\pi}^{(\cdot)}(y(n) ) \thicksim c_\alpha \bar{\pi}^{+}\left( y(n) / \mathbb{E}(\xi^{(\cdot)})  \right),  \quad \text{as } n \rightarrow\infty.
\end{equation}
where $(\cdot)=cn,mn$ and  $c_\alpha= 1/\Gamma( 3 - \alpha)$. 
\end{proposition}
\begin{proof}
We prove the statement for clones, the mutants case is fully similar. First note that in the same way as in the proof of Lemma \ref{convmedidas} (see the Appendix \ref{prn}),
\begin{equation} \label{defmucn}
\mu^{cn}(x)= \int_0^x s \bar{\pi}^{cn}(s) ds, \quad x\geq 0, 
\end{equation}
is a measure on $[0,\infty)$ with Laplace transform $\mathcal{L}_{ \mu^{cn} }$ such that
\begin{equation*} 
\lambda^2 \mathcal{L}_{ \mu^{cn} }(\lambda) = \mathbb{E}\left( 1- e^{-\lambda \xi^{(cn)}} - \lambda \xi^{(cn)}e^{-\lambda \xi^{(cn)}} \right), \quad \lambda\geq 0.
\end{equation*}
We now replace $\lambda$ by a sequence $\{ \lambda(n) : n\geq0 \}$ such that $\lambda(n)\rightarrow0$ as $n\rightarrow\infty$,
\begin{equation*} 
\lambda(n)^2 \mathcal{L}_{ \mu^{cn} }(\lambda(n)) = 1- \phi^c_n(\lambda(n)) +\lambda(n) (\phi^{c}_n)^{\prime} ( \lambda(n) ).
\end{equation*}
An estimate of the term $\phi^c_n$ follows from Lemma \ref{TLlrlclm}
$$
\phi^c_n(\lambda(n)) \thicksim  \phi^{+} ( (1-p(n) )(1-e^{-\lambda(n)}) ), \quad n\rightarrow\infty.
$$
In order to estimate $({\phi}^{c}_n)^{\prime}$ we use the fact that for every fixed $n$, conditionally to $\xi^{(+)}=k$ the distribution of $\xi^{(cn)}$ is Binomial with parameter $(k,1-p(n))$, and we proceed as in the Lemma \ref{TLlrlclm} to get
$$
(\phi^{c}_n)^{\prime}(\lambda(n)) \thicksim   (1-p(n))e^{-\lambda(n)} \phi^{+\prime} ( (1-p(n))(1-e^{-\lambda(n)}) ), \quad n\rightarrow\infty. 
$$
Putting both terms together we infer the estimate,
\begin{align*}
\lambda(n)^2 \mathcal{L}_{\mu^{cn}} ( \lambda(n) )
& \thicksim 1-  \phi^{+} ( (1-p(n) )(1-e^{-\lambda(n)}) )  + (1-p(n))(1-e^{-\lambda(n)} )\phi^{+\prime}_n ( (1-p(n))(1-e^{-\lambda(n)}) ) \\
& \quad - (1-p(n))(1-e^{-\lambda(n)} - \lambda(n)e^{-\lambda(n)}  ) \phi^{+\prime}_n ( (1-p(n))(1-e^{-\lambda(n)}) ),   \quad  \text{ as } n\rightarrow\infty.
\end{align*}
Besides,
$$
(1-p(n))(1-e^{-\lambda(n)} - \lambda(n)e^{-\lambda(n)}  ) \phi^{+\prime}_n ( (1-p(n))(1-e^{-\lambda(n)}) )  \xrightarrow  [n\rightarrow\infty]{}0,
$$
and from (\ref{intasint}) we have
$$
\lambda(n)^2 \mathcal{L}_{\mu} ( \lambda(n) ) = 1 - \phi^{+}(\lambda(n)) +\phi^{+\prime}(\lambda(n)), \quad n \rightarrow \infty,
$$
where $\mathcal{L}_{\mu}$ is the Laplace transform of the measure $\mu$ defined in (\ref{muasint}).
From these last two displays we obtain
$$
\lambda(n)^2 \mathcal{L}_{\mu^{cn}} ( \lambda(n) ) \thicksim ( (1-p(n))(1-e^{-\lambda(n)} ) )^2  \mathcal{L}_{\mu} ( (1-p(n))(1-e^{-\lambda(n)} ) ) + O(\lambda(n)^2), \quad  \text{ as } n\rightarrow\infty.
$$
Due to the estimate $\lambda(n)\thicksim 1-e^{-\lambda(n)}$, in the limit as $n\rightarrow\infty$, the approximation of $\mathcal{L}_{\mu}$ given in (\ref{aproxcolatff}) implies
\begin{equation}  \label{aproxtlcn}
c_\alpha \lambda(n)^2 \mathcal{L}_{\mu^{cn}} ( \lambda(n) ) \thicksim  \bar{\pi}^+ \left( \frac{1}{\lambda(n) (1-p(n))} \right) + O(\lambda(n)^2), \quad n\rightarrow\infty.
\end{equation}

Hence is remains to prove 
\begin{equation} \label{limpic}
\lim_{n\rightarrow\infty} \frac{ \bar{\pi}^{cn}( 1/\lambda(n)) }{ (\lambda(n))^2  \mathcal{L}_{\mu^{cn}} (\lambda(n)) } =c_\alpha.
\end{equation}
In this aim, define for every $y\geq 0$ the following measure
\begin{equation*}
m_{\lambda(n)}^{cn} (0,y] := m_{\lambda(n)}^{cn} (y) = \frac{ \mu^{cn} ( y/ \lambda(n)) }{ \mathcal{L}_{\mu^{cn}} \left( \lambda(n) \right) }.
\end{equation*}
Observe that
\begin{equation*}
\int_{[0,\infty)} e^{ -\theta s } d_y\left(\frac{ \mu^{cn} ( y/ \lambda(n) ) }{ \mathcal{L}_{\mu^{cn}} \left( \lambda(n) \right) } \right)  
= \frac{1}{  \mathcal{L}_{ \mu^{cn} } \left( \lambda(n) \right) } \int_{[0,\infty)} e^{ -\theta  \lambda(n) y } \mu^{cn} (dy)  
= \frac{ \mathcal{L}_{ \mu^{cn} }  \left( \theta \lambda(n) \right) }{  \mathcal{L}_{ \mu^{cn} } \left( \lambda(n) \right) }, \quad  \forall \, \theta>0,
\end{equation*}
from the previous display and (\ref{aproxtlcn}), we get
\begin{equation*}
\mathcal{L}_{ m^{cn}_{\lambda(n)} } ( \theta ) \xrightarrow[n\rightarrow\infty]{} \theta^{-(2-\alpha)}, \quad \forall \theta>0.
\end{equation*}
Writing now $\theta^{-(2-\alpha)}$ in terms of the gamma function we get that
$$
\theta^{-(2-\alpha)} = \frac{1}{\Gamma( 2-\alpha )} \int_0^\infty s^{ (  2-\alpha ) -1 } e^{-\theta s} ds,
$$
and since the convergence of the Laplace ransform implies the weak convergence (see \cite{FellerVol2}), we have
\begin{equation} \label{convmeasm}
m^{cn}_{\lambda(n)}(y) \xrightarrow[n\rightarrow\infty]{} \frac{1}{\Gamma(2-\alpha)} \int_0^y s^{ (2-\alpha) -1 } ds = \frac{ y^{2-\alpha} }{\Gamma(3-\alpha)}.
\end{equation}
Besides, by the definition of the measure $\mu^{(cn)}$ we can obtain that for any $y<1$, the following inequality 
$$
(1/\lambda(n))^2 y(1-y) \bar{\pi}^{cn}(y/\lambda(n)) \leq \mu^{cn}(1/\lambda(n)) - \mu^{cn}(y/\lambda(n)) \leq (1/\lambda(n))^2 (1-y) \bar{\pi}^{cn}( 1/\lambda(n) ),
$$
holds. Now, due to (\ref{convmeasm}), we obtain
$$
 c_\alpha\frac{1-y^{2-\alpha}}{1-y}  \leq \liminf_{n\rightarrow\infty}  \frac{ \bar{\pi}^{cn}( 1/\lambda(n)) }{ (\lambda(n))^2  \mathcal{L}_{\mu^{cn}} (\lambda(n)) }  \leq \limsup_{n\rightarrow\infty} \frac{ \bar{\pi}^{cn}( 1/\lambda(n)) }{ (\lambda(n))^2  \mathcal{L}_{\mu^{cn}} (\lambda(n)) }
\leq c_\alpha \frac{1-y^{2-\alpha}}{y(1-y)} \quad \text{for all } y<1.
$$ 
To conclude we make $y\uparrow 1$. 
\end{proof}

\subsection{Proof of Theorem \ref{Tmainest} } \label{sectpruebas}

Before proving Theorem \ref{Tmainest}, let us we describe the scheme of proof that we will follow. According with the construction of the alleles tree $\{ \mathcal{A}_u:u\in\mathbb{U} \}$, given $u\in\mathbb{U}$ such that $|u|=k\geq1$, for all $j\in\mathbb{N}$ vertices $uj$, represent the size of the $j$-th allelic sub-populations of type $k+1$ begeted by $u$, thus labels of vertices at level $k+1$ determine the variable $T_{k+1}$. Moreover for all $k\in\mathbb{N}$, the total number of vertices at level $k$ correspond to $M_k$. Besides a tree-indexed CSBP is related to the CSBP in the following way: vertices $u\in\mathbb{U}$ at level $|u|=k\geq1$, represent the sizes of the sub-populations at generation $k$ in the $CSBP$ which descend from the same parent at generation $k-1$. So, a first step to prove the converge of alleles trees in Theorem \ref{Tmainest} is to prove the Proposition below.

\begin{proposition}  \label{Pro2estable}
Assuming (\ref{hipcolapit}) and (\ref{hipnew}), we have
\begin{equation*}
\mathcal{L}\left(  \left(  \left(  \frac{ T_k }{ r(n) }, \frac{ M_{k+1} }{ r(n)p(n) } \right):k\in\mathbb{Z}_+ \right),\mathbb{P}_{a(n)}^{p(n)} \right) \Longrightarrow ( ( Z^{1/\alpha}_{k+1}, Z^{1/\alpha}_{k+1}): k\in\mathbb{Z}_+ ), 
\end{equation*}
where $(Z^{1/\alpha}_k:k\in\mathbb{Z}_+)$ is a CSBP process with reproduction measure $\nu^\alpha$ given in Theorem \ref{Tmainest}.
\end{proposition}
 
Since the processes $\{ (T_n,M_{n+1}): n\in\mathbb{Z} \}$ is Markovian with transition probabilities given in (\ref{probatranTM}), to get the previous convergence we need the following Lemma. 
\begin{lemma} \label{L4estable}
For $\alpha\in(1,2)$, let $\tau^\alpha_x$ be an $1/\alpha$-stable subordinator with Laplace exponent 
\begin{equation}\label{LaTrasubalin}
-\log \mathbb{E}(e^{-\lambda \tau^\alpha_x }) = x \int_0^\infty (1- e^{\lambda y}) \nu^\alpha(dy), 
\end{equation}
where $\nu^\alpha$ is as defined in Theorem \ref{Tmainest}. Assuming (\ref{hipcolapit}), (\ref{hip})  and (\ref{hipnew})
\begin{enumerate}
\item [i)] the following convergences holds:
\begin{equation} \label{L4est-i}
\mathcal{L}\left( \left( \frac{T_0}{r(n)}, \frac{ M_{1} } {r(n)p(n)}  \right),\mathbb{P}_{a(n)}^{p(n)} \right) \Longrightarrow (\tau^\alpha_x, \tau^\alpha_x);
\end{equation}
\item [ii)] under the measure $\mathbb{P}_1^{p(n)}$, the behaviour of the joint tail distribution of $(T_0,M_1)$, is given by
\begin{equation} \label{lema4iiest}
\lim_{n\rightarrow \infty} r(n)p(n) \mathbb{P}_1^{p(n)} \left( \frac{T_0}{ r(n)} >s, \frac{ M_{1} }{ r(n)p(n) }>t  \right) = \bar{\nu}^\alpha ( s\wedge t),
\end{equation} 
where $\bar{\nu}^\alpha$ denotes the tail function of the L\'evy measure $\nu^\alpha$. 
\end{enumerate}
\end{lemma}
Observe that the convergence in (\ref{L4est-i}) establishes Poposition \ref{Pro2estable} in the case $k=0$. Because a discrete time CSBP can be constructed from subordinators, that is, given a subordinator $\{\tau_t^\alpha:t\geq 0 \}$, the sequence $\{ Z_k: k\in\mathbb{Z}_+ \}$ defined by iteration as follows:
$$
\sum_{i=1} ^{k+1} Z_i = \tau^\alpha_{x+Z_1+\cdots+Z_k }, \quad k\in\mathbb{Z}_+,
$$ 
with the convention of $Z_0=x$ for some fix $x>0$, is a discrete time CSBP, we refer to \cite{Amaurynotes} for further details. 

The tree-indexed CSBP $\mathcal{Z}$ can also be constructed from this subordinator, namely according to \cite{B2}, $\mathcal{Z}_\varnothing=x$ and $\{ \mathcal{Z}_j : j\in\mathbb{N}  \}$ is given by the sequence of the jumps of $\tau^\alpha$ on $(0,x]$ ranked in decreasing order. The sizes of the subfamilies beget by the subfamily of size $\mathcal{Z}_{1}$ , denoted by $\{ \mathcal{Z}_{1j} : j\in\mathbb{N}  \}$, correspond to the ranked sequences of the jumps of $\tau^\alpha$ on the interval $(x,x+\tau^\alpha_{\mathcal{Z}_1}]$ ranked in decreasing order, $\{ \mathcal{Z}_{2j} : j\in\mathbb{N}  \}$ to those on the interval $(x+\tau^\alpha_{\mathcal{Z}_1},x+\tau^\alpha_{\mathcal{Z}_1+\mathcal{Z}_2}]$. Hence, the convergence (\ref{L4est-i}) will establish the convergence in Theorem \ref{Tmainest} with $u=\varnothing$, once the subordinator $\tau_x^\alpha$, is written in terms of their jumps, or by L\'evy-It\^o decomposition, in terms of the atoms of the Poisson point process associated with it. This is the aim of Lemma below.  




\begin{lemma}\label{nullarrayconex}
Let $b(n)$ be a sequence of integers such that $b(n)\thicksim br(n)p(n)$ for some $b>0$.
\begin{enumerate}
\item [a)]
For every $n\in\mathbb{N}$, let $\{ \chi_j^{(n)} :1\leq j \leq b(n)\}$ be a sequence of independent identically distributed random variables with distribution $ \left( \frac{T_0}{ r(n)}, \frac{ M_{1} }{ r(n)p(n } \right) $. Defining for every $n\in\mathbb{N}$, $\gamma_j^{(n)} := \delta_{\chi_j^{(n)}}$ and $\gamma_n= \sum_{j=1}^{\infty}\gamma_j^{(n)}$, the following convergence holds
\begin{equation} \label{kallth}
 \gamma_n  \Longrightarrow \gamma,
\end{equation}
where $\gamma$ is a Poisson point process with intesity $b\eta$, with $\eta$ the image of the measure $\nu^\alpha$ (given in Theorem \ref{Tmainest}) by the action of the map $x\mapsto (x,x)$. 
\item[b)]
We have the following convergence, under the measure $\mathbb{P}_1^{p(n)}$
\begin{equation}\label{ler} 
\left( \frac{ T_0 }{ r(n) } ,\frac{ M_1 }{ r(n)p(n) } \right)^{ (b(n)^\downarrow) } 
\Longrightarrow 
\left( \mathbf{a}_1, \mathbf{a}_2,...  \right),
\end{equation}
where for all $k\in\mathbb{N}$, $\mathbf{a}_k=(a_k,a_k)$ with $\{a_k\}_{k\in\mathbb{N}}$, the atoms of a Poisson random measure on $(0,\infty)$ with intensity $b\nu^{\alpha}$ ranked in decreasing order; the measure $\nu^{\alpha}$ is given in Theorem \ref{Tmainest}. \end{enumerate}
\end{lemma}
\begin{proof}
From the construction $\{ \gamma_j^{(n)} :1\leq j \leq b(n)\}$ is a sequence of independent random variables. Besides, the convergence (\ref{lema4iiest}) in Lemma \ref{L4estable} implies $ \gamma_j^{(n)}\overset{P}{ \rightarrow } 0$ as $n\rightarrow \infty$, uniformly in $j$. Then 
\begin{equation*} 
\sup_j \mathbb{E}( | \gamma^{(n)}_j \wedge 1  | ) \rightarrow 0.
\end{equation*}
Hence, according with the definition given in chapter 4 of \cite{Kallenberg}, we have that $\{\gamma_j^{(n)} : 1\leq j \leq b(n) \}$ is a null array. Thus we will get the convergence (\ref{kallth}) as an application of Theorem 16.18 of \cite{Kallenberg}, once we verify the following condition:
\begin{enumerate}
\item[i)] $\sum_j \mathbb{P}(\gamma_j^{(n)}( B )> 0) \rightarrow \eta(B)$, as $n\rightarrow \infty$, for all $B\in \hat{\mathcal{B}}$, where $\eta$ is the image of the measure $b\nu^\alpha$ by the application $x\mapsto (x,x),$
\item[ii)] $\sum_j  \mathbb{P}(\gamma_j^{(n)}( B )> 1) \rightarrow 0$, as $n\rightarrow \infty$, for all $B\in \mathcal{B}$.
\end{enumerate}
where $\mathcal{B}$ is the Borel $\sigma$-algebra of $[0,\infty)^{2},$ $\hat{\mathcal{B}} = : \{ B\in\mathcal{B}, \gamma(\partial B)=0 \text{ c.s.} \}$, with $\gamma$ the measure defined in the statement of this Lemma and the symbol $\partial$ denotes the boundary of $B$.
Observe that the class of sets $B = ((b,\infty)\times \mathbb{R}_+)\cup(\mathbb{R}_+\times(b',\infty))$ is a $\pi$-system which generates a $\lambda$-system and it is just $\mathcal{B}$. Then by Dynkin's Theorem, is enough to establish the conditions above for sets of the form $B$. In this setting, the condition (ii) holds because $\gamma_j^{(n)}(\cdot)$ takes only the values $0$ or $1,$ for any $j$ and $n,$ the condition (ii) is true. To establish (i), observe the following equalities  
\begin{align*}
\sum^{b(n)}_{j=1} \mathbb{P}( \gamma^{ (n) }_j (B) > 0 ) 
& = \sum^{ b(n) }_{j=1} \mathbb{P}( \gamma^{ (n) }_j (B) = 1 ) \\
& = \sum^{ b(n) }_{j=1}  \bar{\mu}_n((r(n))^{-1}s,(r(n)p(n))^{-1}t) \\ 
& = b(n) \bar{\mu}_n((r(n))^{-1}s,(r(n)p(n))^{-1}t), 
\end{align*} 
here we recall the identity $\bar{\mu}_n(x,y)=\mathbb{P}_1^{p(n)}(T_0>x \text{ or } M_1>y).$ Assuming that $b(n)\thicksim br(n)p(n)$ for some $b>0$, from Lemma \ref{L4estable} and the last equality we have  
$$
\sum^{b(n)}_{j=1} \mathbb{P}( \gamma^{ (n) }_j (B) > 0 ) 
 \xrightarrow [n\rightarrow \infty]{} b \bar{\nu}^\alpha ( s\wedge t)). 
$$
To get the first convergence in the Lemma, it remains to observe that the equality $b\bar{\nu}^\alpha ( s\wedge t)= \eta(B)$ holds. But this follows from the following equalities, 
$$
\int_B \eta(dx,dy) = b\int_{(x,x)\in B} \nu^\alpha(dx) =b\int_{ ( s\wedge t, \infty)  } \nu^\alpha(dx).
$$
We will now prove the convergence (b). For $i=1,2$, $\chi_{ij}^{(n)}$ denotes the $i$-th coordinate of the sequence $\chi_j^{(n)}$ that appear in the statement (a). Assuming that $\chi_i ^{(n)}:= (\chi_{1i}^{(n)},\chi_{2i}^{(n)}) \leq \chi_j^{(n)}:=(\chi_{1j}^{(n)},\chi_{2j}^{(n)})$ if and only if $\chi_{1i}^{(n)} \leq \chi_{1j}^{(n)}$ or $\chi_{2j}^{(n)} \leq \chi_{2j}^{(n)}$, let us define
$j_1$ as the index where the maximum integer of the sequence $\{ \chi^{(n)}_i : 1\leq i \leq b(n) \}$ is reached.
$$
\chi_{ j_1}^{(n)}=\max_{ 1\leq i \leq b(n)} \chi_i .
$$
Similarly for $k=2,...,b(n)$, let $j_k$ be the index of the $k$-order statistic 
$$ 
\chi^{(n)}_{ j_k}=\max_{ i\in J_k} \chi_i , 
$$ 
where $J_k= \{ 1,..., b(n) \} \setminus \{ j_1,...,j_{k-1} \}$. Then observe that
$$
\mathbb{P}(  \chi_{j_1}^{(n)}  \geq \mathbf{c}_1,  \chi_{j_2}^{(n)} \geq \mathbf{c}_2,...,  \chi_{j_k}^{(n)} \geq \mathbf{c}_{k}  )
=  \mathbb{P}( \gamma_n(C_1)\geq 1, \gamma_n(C_2)\geq 2 , ... , \gamma_n(C_{k}) \geq k ),
$$
if $\mathbf{c}_i=(c_i,c_i)$, $C_i= (0,1)\times(0,1) \setminus (0,c_i) \times (0,c_i) $ and $c_1> \cdots > c_{k} $. Taking now the limit when $n\rightarrow \infty $ in the equality below and using the convergence in (\ref{lema4iiest}), we have  
$$
\mathbb{P}(  \chi_{j_1} \geq \mathbf{c}_1,  \chi_{j_2} \geq \mathbf{c}_2,...,  \chi_{j_k} \geq \mathbf{c}_{k}  )
 \xrightarrow [n\rightarrow \infty]{}  \mathbb{P}( \gamma(C_1)\geq 1,  \gamma(C_2)\geq 2,  ... , \gamma(C_{k}) \geq k ).
$$
This implies the desired convergence because 
$$\mathbb{P}( \gamma(C_1)\geq 1, ... , \gamma(C_{k}) \geq k )= \mathbb{P}(   \mathbf{a}_{j_1}  \geq \mathbf{c}_1,  \mathbf{a}_{j_2} \geq \mathbf{c}_2,..., \mathbf{a}_{j_k} \geq \mathbf{c}_{k}  ),
$$
where $\mathbf{a}_k=(a_k,a_k)$ with $\{a_k\}_{k\in\mathbb{N}}$, the atoms of a Poisson random measure on $(0,\infty)$ with intensity $b\nu^{\alpha}$ ranked in decreasing order; the measure $\nu^{\alpha}$ is given in Theorem \ref{Tmainest}. As before we used the indices $j_k$ to rank in decreasing order the sequence $\mathbf{a}_k$.
\end{proof}

\subsubsection{Proof of Lemma \ref{L4estable}}
For each $n\in\mathbb{N}$, let $\mathbf{S}_k^{(n)}$ be a random vector such that their coordinates are obtained from $\xi^{(cn)}$ and $\xi^{(mn)}$, as follows
\begin{equation*}
\mathbf{S}_k^{(n)} = (a(n),0) +  \sum_{i=1}^{k} \left( \xi_i^{(cn)}-1,  \xi_i^{(mn)}   \right), \quad k\in\mathbb{N},
\end{equation*}
where
$\{ (\xi_i^{(cn)}, \xi_i^{(mn)} ): i\in\mathbb{N}\}$ are a pair of i.i.d. random variables with distribution $ (\xi^{(cn)}, \xi^{(mn)} )$. 
  
A key result in the proof of Lemma  \ref{L4estable} is the following: 
\begin{proposition}\label{convfunsta}
In the regime (\ref{hipcolapit}), (\ref{hip}) and (\ref{hipnew}) the normalized random walk defined by
$$
\bar{\mathbf{S}}_{\lfloor r(n)t \rfloor}^{(n)} = (a(n)/r(n)p(n),0) +  \sum_{i=1}^{\lfloor r(n)t \rfloor} \left( (\xi_i^{(cn)} -1)/n,  \xi_i^{(mn)} / r(n)p(n) \right), \quad t\geq0,
$$
converges weakly
\begin{equation*} \label{limconvfunsta}  
 \left(  \bar{ \mathbf{S} }_{\lfloor r(n)t \rfloor }^{(n)} : t\geq 0 \right) 
 \Longrightarrow 
\left(  \left(x+X_t, t \right)  : t\geq 0 \right),
\end{equation*}
where $\{X_t,t\geq0\}$ is an $\alpha$-stable process with no-negative jumps with and characteristic exponent $c_\alpha |\lambda|^\alpha$.
\end{proposition}
We pospone the proof of this result to the appendix and focus on the proof of Lemma \ref{L4estable}.
\begin{proof}[Proof of Lemma \ref{L4estable} i) ]
From Lemma 3 of \cite{B2}, we know that the first passage time below $0$ for the centered random walk $\mathbf{S}_{1,k}^{(n)}= a(n) + \sum_{i=1}^{k} (\xi_k^{(cn)} -1)$ has the same distribution as $T_0$. Moreover, on the one hand, because of the identities,
$$
\varsigma^{(n)}(0) =  \inf \{ k\in\mathbb{Z}_+ : \mathbf{S}_{1,k}^{(n)}=0 \}  \quad \text{and} \quad \Sigma^{(n)}(0)= \sum_{i=1}^{\varsigma^{(n)}(0)} \xi_i^{(mn)},
$$
the random variables 
\begin{equation*}
(\varsigma^{(n)}(0), \Sigma^{(n)}(0))   \quad \text{and} \quad (T_0,M_1),
\end{equation*}
have the same distribution under $\mathbb{P}_{a(n)}^{p(n)}$. On the other hand, we also have the following two identities 
$$
\frac{\varsigma^{(n)}(0)}{r(n)}=\frac{1}{r(n)}\inf\{ k\in\mathbb{Z}_+: \mathbf{S}_{1,k}^{(n)}=0 \} = \inf\{ t\geq 0: \bar{ \mathbf{S} }_{1,\lfloor r(n)t \rfloor }^{(n)}=0 \}.
$$ 
and 
\begin{equation} \label{limca} 
\left ( \frac{1}{r(n)} \, \varsigma^{(n)}(0), \bar{\mathbf{S}}^{(n)}_{ \varsigma^{(n)}(0) } \right) =
\left ( \frac{1}{r(n)} \, \varsigma^{(n)}(0), \left( \bar{\mathbf{S}}_{1, \lfloor \varsigma^{(n)}(0) \rfloor}^{(n)}   , \frac{1}{r(n)p(n)} \, \Sigma^{(n)}(0) \right) \right),
\end{equation}
and the weak convergence
\begin{equation*} \label{limconvfunsta}
\left(  \bar{ \mathbf{S} }_{\lfloor r(n)t \rfloor }^{(n)} : t\geq 0 \right)  
\Longrightarrow 
\left(  \left(x+X_t, t \right)  : t\geq 0 \right).
\end{equation*}
Since $X$ is a spectrally positive L\'evy process, Theorem 1 in Chapter VII of \cite{BertoinLevy} ensures that the first passage time below $-x$ for the process $X$  
\begin{equation*}
\tau_x^\alpha= \inf\{ t\geq 0:  X_t \leq -x \}, \quad x\geq 0.
\end{equation*}
is a stable subordinator of parameter $1/\alpha$. 
We conclude from these facts that the claimed convergence holds as soon as
\begin{equation} \label{convnota1corrFeb}
\left ( \frac{1}{r(n)} \, \varsigma^{(n)}(0), \bar{\mathbf{S}}^{(n)}_{ \varsigma^{(n)}(0) } \right)
 \Longrightarrow 
\left( \tau_x^\alpha, (x+ X_t, t )| _{t=\tau^\alpha_x}  \right).
\end{equation} 
But according to Theorem 13.6.5 of \cite{Whitt} about weak convergence of first passage times and under shoots and overshoots, when there is convergence in Skorohod's topology, we get that 
$$
\left ( \frac{1}{r(n)} \, \varsigma^{(n)}(0) , \bar{\mathbf{S}}_{1, \varsigma^{(n)}(0) }^{(n)}  \right) 
 \Longrightarrow 
 \left( \tau_x^\alpha, X_{\tau_x^\alpha}+ x\right)
$$
Moreover since we have the joint convergence 
$$
\left( \bar{\mathbf{S}}_{1, \lfloor \varsigma^{(n)}(0)t \rfloor}^{(n)},\bar{ \mathbf{S} }_{2, \lfloor \varsigma^{(n)}(0) t \rfloor}^{(n)}: t\geq 0 \right  )  \Longrightarrow  \left( ( x+ X_t, t): t\geq 0 \right)
$$
in the sense of Skorohod's topology, and the second coordinate is a determinist linear function, its is an elementary exercise to extend the above mentioned result of \cite{Whitt} to get that the convergence in (\ref{convnota1corrFeb}) holds. 
\end{proof}

\begin{proof}[Proof of Lemma \ref{L4estable} ii) ]
We will apply the same techniques used in the proof of statement (ii) in Lemma 4 of \cite{B2}. Let us start observing that for every $x,y\in\mathbb{R}$ 
$$ 
e^{-sx-ty} = st \int_0^\infty \int_0^\infty  e^{-sx-ty} \mathbf{1}_{\{ x<u, y <v\}} dudv, \quad s,t\geq 0. 
$$
Thus Fubini's Theorem implies that for any random vector $(X,Y)$ the following identity holds.
$$
1-\mathbb{E}(e^{-sX-tY}) =  st \int_0^\infty \int_0^\infty  e^{-sx-ty} \mathbb{P}(X\geq u \text{ or }Y\geq v) dudv, \quad s,t\geq 0.
$$
In particular,
$$
1-\mathbb{E}_1^{p(n)}(e^{-\frac{s}{r(n)}T_0-\frac{t}{r(n)p(n)}M_1}) = st \int_0^\infty  \int_0^\infty e^{-su-tv} \bar{\mu}_n((r(n))^{-1}u,(r(n)p(n))^{-1}v)dudv, \quad s,t\geq 0,
$$
where $\bar{\mu}_n(x,y):=\mathbb{P}_1^{p(n)}(T_0>x \text{ or } M_1>y)$.
Hence by the branching property, 
\begin{equation} \label{ecprelim}
\mathbb{E}_{a(n)}^{p(n)}(e^{-\frac{s}{r(n)}T_0-\frac{t}{r(n)p(n)}M_1})  = \left(1-st \int_0^\infty  \int_0^\infty e^{-su-tv}  \bar{\mu}_n((r(n))^{-1}u,(r(n)p(n))^{-1}v) dudv \right)^{a(n)}.
\end{equation}
Now, according to (i) the left term in the previous display converges as $n\rightarrow\infty$ towards 
$$
\mathbb{E}\left( e^{ - (s+t)\tau_x^{\alpha} }  \right).  
$$
Besides the right term converges as $n\rightarrow\infty$ towards 
$$
 \exp \left( -x\int_0^\infty( 1-e^{- (s+t)y } )\nu^{\alpha}(dy) \right),
$$
here we use (i) in the Lemma \ref{L4estable}, together with (\ref{hip}). Therefore
\begin{equation} \label{limtL01anest}
\mathbb{E}\left( e^{ - (s+t)\tau_x^{\alpha} }  \right) 
= \exp \left( -x\int_0^\infty( 1-e^{- (s+t)y } )\nu^{\alpha}(dy) \right),
\end{equation} 
Taking logarithms in (\ref{ecprelim}) and (\ref{limtL01anest}) we obtain
$$
\lim_{n\rightarrow\infty} sta(n)\int_0^\infty\int_0^\infty e^{-su-tv}\bar{\mu}_n((r(n))^{-1}u,(r(n)p(n))^{-1}v)dudv  
= x \int_0^\infty (1-e^{-(s+t)y})\nu^{\alpha}(dy).
$$
Hence it only remains to see that the line above is equal to 
$$
xst\int_0^\infty\int_0^\infty e^{-sy-tz}\bar{\nu}^\alpha(y\wedge z)dydz.
$$ 
For that end we observe the equality
$$
\int_0^\infty\int_0^\infty e^{-sy-tz}\bar{\nu}^\alpha(y\wedge z)dydz
= \int_0^\infty \nu^{\alpha}(du) \int_0^\infty \int_0^\infty e^{-sy-tz} \mathbf{1}_{ \{ u>y \text{ or }  u>z\} }dydz, 
$$
and we obtain the claimed identity by uniqueness of the Laplace transform. 
\end{proof}

\subsubsection{Proof of Proposition \ref{Pro2estable}}
Observe it is enough to show the convergence of the Laplace transforms associated with the finite dimensional distributions of each processes. So, we will briefly deduce the Laplace transform of a CSBP (in discrete time) $\{Z_k:k\in\mathbb{Z}_+\}$. From \cite{Amaurynotes}, we know that the transition probabilities of this branching process are characterized, for every $k\in\mathbb{Z}_+$ and $\lambda,z\geq0$ by
\begin{equation} \label{chftcsbpdisc}
\mathbb{E}( e^{ - \lambda Z_{k+1} } | Z_{k} =z  ) = e^{-z\kappa(\lambda)}.
\end{equation}
where $\kappa$ is the cumulant of a subordinator without drift, that is
\begin{equation} \label{kappa}
\kappa(\lambda)= \int_{(0,\infty)} (1 - e^{-\lambda x})\vartheta(dx),
\end{equation}
where $\vartheta$ is the L\'evy measure. Applying successively the property (\ref{chftcsbpdisc}), we obtain 
$$
\mathbb{E}_x( e^{-s_1Z_1 - \cdots - s_k Z_k} ) =  e^{ -x \kappa ( l_{k-1}( s_1)  )  } \qquad s_i\geq 0, \,i=1,2,...,k.
$$
where by definition $l_0(s)=s$ and
\begin{equation} \label{ele}
l_{i}(s_{n-i})= s_{n-i}+\kappa(l_{i-1}(s_{n-i+1}))), \quad i\in\mathbb{Z}_+.  
\end{equation}
In particular, $\{ Z_k^{1/\alpha} : k\in\mathbb{Z}_+  \}$ is a CSBP whose transition probabilities are characterized by a subordinator $\tau_x^\alpha$, with cumulant given by (\ref{kappa}) with $\vartheta= \nu^\alpha$ defined in Theorem \ref{Tmainest}. Therefore we will prove by induction on $k$ the following convergence: 
\begin{equation} \label{limfindimks}
\mathbb{E}_{a(n)}^{p(n)}\left( \prod_{i=1}^k e^{ -\frac{s_{i-1}}{r(n)} T_{i-1}-\frac{t_i}{r(n)p(n)} M_i } \right)
\xrightarrow[n\rightarrow \infty] {}
\exp\{-x\kappa(l_{k-1}(s_0+t_1))\} , \quad \text{for all } s_i,t_i\geq 0, i=1,2,...,k.
\end{equation}
where $l$ is defined in (\ref{ele}), taking $\vartheta= \nu^\alpha$. We start observing that $Z_1^{1/\alpha}$ has the same distribution of $\tau_x^\alpha$ because of CSBP's construction. Thus the case $n=1$ follows from Lemma \ref{L4estable} (i). We next assume (\ref{limfindimks}) holds for $n=k$ and prove the convergence for $n=k+1$.  To this end, we use the Markov property of $\{ (T_n,M_{n+1}): n\in\mathbb{Z} \}$ and the fact that, conditionally to $M_n=j$, the pair $(T_n,M_{n+1})$ has the same distribution of $(T_0,M_1)$ under $\mathbb{P}_j$, to obtain 
\begin{align*}
\mathbb{E}_{a(n)}^{p(n)}&\left( \prod_{i=1}^{k+1} e^{ -\frac{s_{i-1}}{r(n)} T_{i-1}-\frac{t_i}{r(n)p(n)} M_i } \right) \\
& =  \mathbb{E}_{a(n)}^{p(n)} \left( e^{ -\frac{s_0}{r(n)} T_0 -\frac{t_1}{r(n)p(n)} M_1} \cdots  
    e^{-\frac{s_{k-1}}{r(n)}T_{k-1} -\left( \frac{t_k}{r(n)p(n)} - \frac{1}{r(n)p(n)} \log \mathbb{E}_{ r(n)p(n) }^{p(n)} \left( e^{ -\frac{s_k}{r(n)} T_0 -\frac{t_{k+1}}{r(n)p(n)} M_1 }\right)                           \right)M_k } \right).
\end{align*}   
Due to the assumption $r(n)p(n)\thicksim a(n) x$ in hypothesis (\ref{hipnew}), we obtain as consequence of Lemma \ref{L4estable} (i), that 
\begin{equation*}
\mathbb{E}_{a(n)}^{p(n)}\left( \prod_{i=1}^{k+1} e^{ -\frac{s_{i-1}}{r(n)} T_{i-1}-\frac{t_i}{r(n)p(n)} M_i } \right) 
 \thicksim  
    \mathbb{E}_{a(n)}^{p(n)} \left( e^{ -\frac{s_0}{r(n)} T_0 -\frac{t_1}{r(n)p(n)} M_1} \cdots  
    e^{-\frac{s_{k-1}}{r(n)}T_{k-1} -\frac{1}{r(n)p(n)}\left( t_k + \kappa(s_k+t_{k+1}) \right)M_k } \right).
\end{equation*}
Then using the induction hypothesis with
\begin{equation} \label{stprima}
s'_{i-1}+t'_i=
\left\{
\begin{array}{rl}
s_{i-1}+t_i & i<k, \\
l(s_{i-1}+t_i) &  i=k, \\
\end{array}
\right. 
\end{equation}
we get
\begin{multline*}
\mathbb{E}_{a(n)}^{p(n)} \left( e^{ -\frac{s_0}{r(n)} T_0 -\frac{t_1}{r(n)p(n)} M_1} \cdots  
      e^{-\frac{s_{k-1}}{r(n)}T_{k-1} -\frac{1}{r(n)p(n)}\left( t_k +  \kappa(s_k+t_{k+1}) \right)M_k } \right) \\
\xrightarrow[n\rightarrow \infty] {}
\exp\{-x\kappa(l_{k-1}(s'_0+t'_1))\}.    
\end{multline*} 
This concludes the proof because of the recursive definition of $l_i$ given in (\ref{ele}) together with the choice of $s'_{i-1}+t'_i$,
$$ 
\kappa(l_{k-2}(s'_1+t'_2)) =  \kappa(l_{k-1}(s_1+t_2)),
$$
as consequence of $l_{k-1}(s'_0+ct'_1) = l_k(s_0+t_1)$.  \begin{flushright}$\blacksquare$\end{flushright}

\subsubsection{Proof of Theorem \ref{Tmainest}}
We will establish
\begin{equation}\label{convfindimgen}
\mathcal{L}\left(  \left(   \left(   \left(  \frac{\mathcal{A}^{(n)}_u}{r(n)}, \frac{d^{(n)}_u}{r(n)p(n)}   \right) :  |u| \leq k  \right) : k\in\mathbb{Z}_+   \right) , \mathbb{P}_{a(n)}^{p(n)}  \right ) \Longrightarrow
\left( \left( \left( \mathcal{Z}_u, \mathcal{Z}_u \right) : |u| \leq k \right)  :  k\in\mathbb{Z}_+ \right). 
\end{equation}
Actually by the monotone class theorem, it is enough to show for non-negative measurable continuous functions $f_1,...,f_k$  
$$
\mathbb{E}_{a(n)}^{p(n)}  \left[  \prod_{i=1}^k  f_i( ( r(n) )^{-1} \mathcal{A}^{(n)}_{u},(r(n)p(n))^{-1}d^{(n)}_{u}  : |u| \leq i ) \right]
\xrightarrow[n\rightarrow \infty] {}
\mathbb{E}^{\mathbb{Q}_x }  \left[  \prod_{i=1}^k  f_i( \mathcal{Z}_{u},\mathcal{Z}_{u} : |u| \leq i ) \right], 
$$
where $\mathbb{Q}_x$ is the law of a tree-indexed CSBP started with an initial population of size $x$ contructed from a subordinator $\{ \tau_t^\alpha: t\geq 0 \}$. Let us prove it by induction on $k$. The case $k=1$ is given in the convergence (\ref{ler}). Assuming the case $k$, we will prove the convergence for $k+1$. Conditioning with respect to $\mathcal{F}_k= \sigma( \mathcal{A}^{(n)}_u,d^{(n)}_u :  | u |  \leq  k )$ and using Lemma \ref{L2B2} we have  
\begin{align*}
\mathbb{E}_{a(n)}^{p(n)}& \left[  \mathbb{E}_{a(n)}^{p(n)}  \left(  \prod_{i=1}^{k+1}  f_i( ( r(n) )^{-1} \mathcal{A}^{(n)}_{u},(r(n)p(n))^{-1}d^{(n)}_{u}  : |u| = i ) \left|   \right.  \mathcal{F}_k \right) \right] \\
& \qquad \qquad = \mathbb{E}_{a(n)}^{p(n)}   \left[ \prod_{i=1}^k   \right. f_i( ( r(n) )^{-1} \mathcal{A}^{(n)}_{u},(r(n)p(n))^{-1}d^{(n)}_{u}  : |u| \leq i ) \\
& \qquad \qquad \qquad \qquad \left. \mathbb{E}_{1}^{p(n)}  \left(  f_{k+1}( ( (r(n))^{-1} T^{(n)}_0,(r(n)p(n))^{-1}M^{(n)}_1 )^{ d_u^{(n)}\downarrow}  : |u| =k )  \right)   \right] 
\end{align*}
Besides, by induction hypothesis $d_u^{(n)}\thicksim r(n)p(n)\mathcal{Z}_u$ with $|u|=k$, therefore when $n\rightarrow \infty$ in the previous equality we obtain
$$
\mathbb{E}^{\mathbb{Q}_x }  \left[  \prod_{i=1}^k  f_i( \mathcal{Z}_{u},\mathcal{Z}_{u} : |u| \leq i ) \mathbb{E}_{1} (f_{k+1}( \left( \mathbf{a}'_1, \mathbf{a}'_2,...  \right)  ) ) \right] ,
$$
where $\mathbf{a}'_k=(a'_k,a'_k)$ are the atoms of a Poisson random measure on $(0,\infty)$ with intensity $b \mathcal{Z}_{u}\nu^{\alpha}$, repeated according to their multiplicity and ranked in the decreasing order. Due to the definition of a CSBP tree this concludes the proof. 
\begin{flushright}$\blacksquare$\end{flushright}

\subsection{The conditioned to non-extinction case}
This section is devoted to prove Theorem \ref{T1}. Following the same strategy of Proposition \ref{Pro2estable}, we shall establish by induction the convergence of Laplace transforms of the finite dimensional distributions associated with the processes involved. In this aim, we firstly deduce the Laplace transform of the finite dimensional distributions of a CSBP with immigration, with mechanism $(\vartheta, \iota)$, $\{ Z^I_n : n\in\mathbb{N} \}$, defined for every $n\in\mathbb{N}$ as follows 
$$
Z^I_{n+1} = \tau_n(Z^I_n) + I_n, 
$$ 
where $\{I_n:n\in\mathbb{N}\}$ is a sequence of nonnegative random variables with common probability measure $\omega$, which determine the distribution of individual immigrants arriving in the population. Let us denote its Laplace transform of $\omega$ by $\iota$, i.e.
\begin{equation}\label{inmlt}
\iota(\lambda)=\int_0^\infty e^{-\lambda x}\omega(dx), \quad \lambda\geq0.
\end{equation}
Let $\{T_n(t):t\geq0\}_{n\geq0}$ be a sequence of independent subordinators (without drift) and also independents of $I_n$, with the same distribution and Laplace transform given in (\ref{kappa}). Thereby
$$
\mathbb{E} ( e^{-\lambda  T_n(Z^I_n) } | Z^I_n ) = e^{ - Z^I_n \kappa(\lambda) }.
$$
This previous equality together with the Markov property imply that
$$
\mathbb{E}_x( e^{-s_1Z^I_1 - \cdots - s_k Z^I_k} ) =  \prod_{i=0}^{k-1} \iota( l_i( s_{k-i} ) )e^{ -x \kappa( l_{k-1} (s_1) ) }, \quad \text{for all } s_i\geq 0, i=1,2,...,k.
$$
Thus the proof of statement (i) of Theorem \ref{T1} require to establish for all $s_i,t_i\geq 0$, $i=1,2,...,k$, the convergence below holds
\begin{equation} \label{limfindimk}
\mathbb{E}_{a_n}^{p(n)\uparrow}\left( \prod_{i=1}^k e^{ -\frac{s_{i-1}}{n^2} T_{i-1}-\frac{t_i}{n} M_i } \right)
\xrightarrow[n\rightarrow \infty] {}
\exp\{-x\kappa(l_{k-1}(s_0+c_\vartheta t_1))\} \prod_{i=0}^{k-1} \iota(l_{k-i}(s_{i-1}+\tilde{c}_\vartheta t_i)),
\end{equation} 
where $\kappa$ and $l$ are respectively defined in (\ref{kappa}) and (\ref{ele}), taking in partiular $\vartheta=\nu$ given in (\ref{fdpnu}). To obtain (ii) of Theorem \ref{T1}, the previous convergence is proved with $\vartheta=\nu^\alpha$ defined in (\ref{nuest}). The following Lemma establishes the above convergence in the case $k=1$. In their proof we use the reference \cite{B2} to justify the statement corresponding to $\vartheta=\nu$ and previous results for the associated to $\vartheta=\nu^\alpha$. 
\begin{lemma} \label{LconvT0M1nextsub}
If (\ref{hip}) holds, then we have the following convergence 
\begin{equation*} 
\mathcal{L}\left( ( b_1(n)T_0,b_2(n)M_{1} ),\mathbb{P}_{a(n)}^{p(n)\uparrow} \right) \Longrightarrow (\tau, c_\vartheta \tau), 
\end{equation*}
where $\tau$ is a random variable with Laplace transform $e^{-\kappa(s)}\iota(s)$, where $\kappa(s)$, $\iota(s)$ are given in (\ref{kappa}) and (\ref{inmlt}), according to
\begin{enumerate}
\item [a)] if the reproduction law has finite variance $\sigma^2$, $\vartheta=c^{-1}\nu$, where $\nu$ is the measure in (\ref{fdpnu}). Moreover $b_1(n)= n^{-2}$, $b_2(n)=n^{-1}$ and $c_\vartheta=c$;
\item [b)] otherwise, under the assumptions (\ref{hipcolapit}) and (\ref{hipnew}),
$b_1(n)= (r(n))^{-1}$, $b_2(n)=(r(n)p(n))^{-1}$,  $\vartheta$ is given by (\ref{nuest}) and $c_\vartheta=1$.
\end{enumerate}
\end{lemma}
\begin{proof} 
To simplify the notation we just write $b_1$ and $b_2$. We prove the convergence of Laplace transform $( b_1T_0,b_2 M_{1} )$ under the measure $\mathbb{P}_{a(n)}^{p(n)\uparrow} $. First, recalling the definition of the conditional measure given in (\ref{pflecha}), an elementary calculation using the branching property implies
\begin{equation} \label{tLflecha01}
\mathbb{E}_a^{p\uparrow}\left( e^{ -s T_0-t M_1 } \right)
 = \mathbb{E}_{a-1}^p \left( e^{ -s T_0-t M_1 }   \right) \mathbb{E}_1^p \left( e^{ -s T_0-t M_1 } M_1 \right).
\end{equation}
Thanks to Lemma 4 of \cite{B2} and Lemma \ref{L4estable}.
\begin{equation} \label{limtL01an} 
\mathbb{E}_{a(n)}^{p(n)}\left( \exp \left( - sb_1T_0 - tb_2M_1  \right) \right) 
\xrightarrow[n\rightarrow \infty] {}
\exp \left( -x\int_0^\infty( 1-e^{- (s+c_\vartheta t)y } )c_\vartheta^{-1}\vartheta(dy) \right),
\end{equation}
so it remains to calculate the limit of the second factor in (\ref{tLflecha01}). Using again \cite{B2} together with the equality (\ref{ecprelim}), we have
\begin{equation*} \label{tL01}
\mathbb{E}_1^{p(n)}\left( \exp \left( - sb_1T_0 - tb_2 M_1  \right) \right) 
= 1-st\int_0^\infty \int_0^\infty e^{-sx-ty}\bar{\mu}_n(  b_1^{-1}x, b_2^{-1}y)dxdy, 
\end{equation*}
as before $\bar{\mu}_n(x,y):=\mathbb{P}_1^{p(n)}(T_0>x \text{ or } M_1>y)$. This implies due to Lemma 4 (ii) in \cite{B2} and Lemma \ref{L4estable} (ii) 
\begin{align} \label{qinmigracion}
\mathbb{E}_1^{p(n)}\left( \exp \left( - sb_1T_0 - t b_2 M_1  \right)M_1 \right) 
\xrightarrow[n\rightarrow \infty] {} &
s\int_0^\infty \int_0^\infty e^{-sx-ty}  \bar{\vartheta}\left(x\wedge \frac{y}{c_\vartheta} \right) dxdy \nonumber \\
& \quad \quad \quad \quad -st\int_0^\infty \int_0^\infty ye^{-sx-ty}\bar{\vartheta}\left(x\wedge \frac{y}{c_\vartheta} \right) dxdy.  \nonumber 
\end{align}
Computing the integrals we get
\begin{equation} \label{inmigracion}
\mathbb{E}_1^{p(n)}\left( \exp \left( - sb_1T_0 - tb_2 M_1  \right)M_1 \right) 
\xrightarrow[n\rightarrow \infty] {}  \int_0^\infty  e^{-(s+c_\vartheta t)z} z \vartheta(dz).
\end{equation}
This finishes the proof.
\end{proof}
We can now continue with the proof of Theorem \ref{T1}, assuming that (\ref{limfindimk}) holds for $k$ and to verify the case $k+1$. Let be $\mathcal{F}_k=\sigma((M_{j-1},T_j),j\leq k)$ and $(T'_0,M'_1)$ an independent copy of $(T_0,M_1)$. Recalling the definition of the measure $\mathbb{P}^\uparrow$ we have, 
\begin{align*}
\mathbb{E}_{a}^{p\uparrow}& \left( e^{ -\lambda_0 T_0 -\theta_1 M_1 } \cdots e^{ -\lambda_k T_k -\theta_{k+1} M_{k+1} } \right) \\
& \qquad \qquad\qquad\qquad\qquad\quad= \mathbb{E}_{a}^{p} \left( e^{ -\lambda_0 T_0 -\theta_1 M_1 } \cdots e^{ -\lambda_{k-1} T_{k-1} -\theta_k M_k } \frac{1}{a} 
                   \mathbb{E}_{M_k}^{p} \left( e^{ -\lambda_k T'_0 -\theta_{k+1} M'_1 } M'_1  \right) \right).   
\end{align*}        
Then applying the identity (\ref{tLflecha01}) we get           
\begin{align*} 
\mathbb{E}_{a}^{p\uparrow}& \left( e^{ -\lambda_0 T_0 -\theta_1 M_1 } \cdots e^{ -\lambda_k T_k -\theta_{k+1} M_{k+1} } \right)  \\                 
& =\mathbb{E}_{a}^{p} \left( e^{ -\lambda_0 T_0 -\theta_1 M_1 } \cdots e^{ -\lambda_{k-1} T_{k-1} -\theta_k M_k } \frac{M_k}{a} 
    \left( \mathbb{E}_1^{p} \left( e^{ -\lambda_k T'_0 -\theta_{k+1} M'_1 } \right) \right)^{M_k-1} 
    \mathbb{E}_1^{p} \left( e^{ -\lambda_kT'_0-\theta_{k+1} M'_1 } M'_1 \right) \right) .
\end{align*}
Now using the Markov property of $\{ (T_n,M_{n+1}),n\in\mathbb{Z}_+ \}$ and writing the terms suitably, we get
\begin{align*}
\mathbb{E}_{a}^{p\uparrow}&\left( e^{ -\lambda_0 T_0 -\theta_1 M_1 } \cdots e^{ -\lambda_k T_k -\theta_{k+1} M_{k+1} } \right) \\
& = e^{- \frac{1}{n} \log \mathbb{E}_n^{p} \left( e^{ -\lambda_k T'_0 -\theta_{k+1} M'_1 } \right) } 
    \mathbb{E}_1^{p} \left( e^{ -\lambda_kT'_0-\theta_{k+1} M'_1 } M'_1 \right) \\
& \qquad \qquad \qquad \qquad \qquad \quad 
  \times  \mathbb{E}_a^{p} \left( e^{-\lambda_0 T_0 -\theta_1 M_1 } \cdots e^{ -\lambda_{k-1} T_{k-1}} 
    e^{-\left( \theta_k - b_2 \log \mathbb{E}_{b^{-1}_2}^{p} \left( e^{ -\lambda_k T'_0 -\theta_{k+1} M'_1 } \right) \right)M_k } \frac{M_k}{a} \right). 
\end{align*}
\noindent Taking in the previous equality $(\lambda_{i-1},\theta_{i})=( b_1s_{i-1}, b_2 t_i)$, $i=1,...,k+1$ and $(a,p)=(a(n),p(n))$, we use hypothesis (\ref{hip}) and (\ref{hipnew}) to obtain 
\begin{align*}
\mathbb{E}_{a(n)}^{p(n)\uparrow}&\left( \prod_{i=1}^{k+1} e^{ - b_1 s_{i-1}T_{i-1}- b_2 t_i M_i } \right) \\
& \thicksim  e^{- b_2 \log \mathbb{E}_{b_2^{-1}}^{p(n)} \left( e^{ - b_1s_kT'_0 - b_2 t_{k+1} M'_1 } \right) }  
    \mathbb{E}_1^{p(n)} \left( e^{ - b_1s_kT'_0- b_2 t_{k+1} M'_1 } M'_1 \right) \\ 
&  \qquad \qquad \qquad \qquad \times
    \mathbb{E}_{a(n)}^{p(n)\uparrow} \left( e^{ - b_1 s_0 T_0 - b_2 t_1 M_1} \cdots  
    e^{- b_1s_{k-1}T_{k-1} - b_2\left( t_k +  \kappa(s_k+ct_{k+1}) \right)M_k } \right). 
\end{align*}
Now we have to calculate the limit of each factor. The first one converges towards to $1$ thanks to (\ref{limtL01an}). Besides, to get  
\begin{equation} \label{limfact2}
\mathbb{E}_1^{p(n)} \left( e^{ - b_1 s_k T'_0 - b_2t_{k+1} M'_1 } M'_1 \right)
\xrightarrow[n\rightarrow \infty] {}
\iota(\kappa_0(s_k+c_\vartheta t_{k+1})),
\end{equation}
we use the convergence (\ref{inmigracion}) together the convention $\kappa_0(s)=s$. As in Proposition \ref{Pro2estable}, in order to conclude we use the induction hypothesis with $s'_{i-1}+t'_i$, $1\leq i \leq k$ defined in (\ref{stprima}).

\appendix
\section{Proof of Lemma \ref{convmedidas}} \label{prn}
A consequence of (\ref{hipcolapit}) is that the measure defined on $[0,\infty)$ by
\begin{equation} \label{muasint}
\mu(x):= \int_0^x z\bar{\pi}^{+}(z)dz, \quad x\geq 0. 
\end{equation}
is such that $x\mapsto \mu(x)$ is $RV_\infty^{2-\alpha}$. Then from the Tauberian-Abelian Theorem (see \cite{Teugles}), its Laplace transform  $\mathcal{L}_\mu\in RV_ 0^{-(2-\alpha)}$ and 
\begin{equation*} \label{mlmta}
\mu(x) \thicksim \frac{1}{\Gamma(3-\alpha)} \mathcal{L}_\mu \left( 1/x \right), \quad x\rightarrow\infty.
\end{equation*}
Observe that  
\begin{equation*}
\lambda^2 \mathcal{L}_\mu(\lambda) = \int_0^\infty ( 1-e^{-\lambda y} -\lambda y )\pi^{+}(dy) +  \int_0^\infty \lambda z(1-e^{-\lambda z})\pi^{+}(dz).
\end{equation*}
Then
\begin{equation} \label{intasint}
\lambda^2 \mathcal{L}_\mu(\lambda) = \mathbb{E}\left( 1- e^{-\lambda \xi^+} - \lambda \xi^+ e^{-\lambda \xi^+} \right), \quad \lambda \rightarrow 0.    
\end{equation}
As consequence of the definition of the measure $\mu$ and the approximations above,
\begin{equation} \label{aproxcolatff} %
\bar{\pi}^{+}\left( 1/\lambda \right) \thicksim c_\alpha \mathbb{E}\left( 1- e^{-\lambda \xi^+} - \lambda \xi^+ e^{-\lambda \xi^+}\right), \quad \lambda \rightarrow 0,
\end{equation}
where $c_\alpha=1/\Gamma(3-\alpha)$. Hence for all $x>0$,
\begin{equation} \label{fspsiasint1} 
\frac{1}{\mathbb{E}\left( 1- e^{-\lambda \xi^+} - \lambda \xi^+ e^{-\lambda \xi^+}\right )} \bar{\pi}^{+} \left( x/\lambda \right) \thicksim  \frac{ \bar{\pi}^+ \left( x/\lambda \right) }{ c_\alpha \bar{\pi}^+\left( 1/\lambda \right) }  \xrightarrow[\lambda \rightarrow 0] {} c_\alpha x^{-\alpha}.
\end{equation}
We set $r(n)=\left(\mathbb{E}\left[ 1- e^{- \xi^+/n} - \xi^+ e^{- \xi^+/n}/n\right] \right)^{-1}$ and define the measure on $(0,\infty)$, $m_n(dy) = r(n)\bar{\pi}^{+}(n dy)$. The convergence in (\ref{fspsiasint1}) implies 
$$
m_n(x,\infty) \xrightarrow[n \rightarrow \infty] {}  \int_x^\infty \frac{c_\alpha}{\alpha} \frac{dy}{y^{1+\alpha}}, \quad \text{for all } x\geq 0.
$$
Therefore, for all $0<x\leq y \leq \infty$
$$
m_n(x,y] \xrightarrow[n \rightarrow \infty] {}  \int_x^y \frac{c_\alpha}{\alpha} \frac{dz}{z^{1+\alpha}}. 
$$
This implies that the measure on $(0,\infty)$ defined by $m_n(dy) = r(n)\bar{\pi}^{+}(n dy)$ converges vaguely towards $c_\alpha \frac{dy}{y^{1+\alpha}}$. We also have  
\begin{equation*} \label{finlabel}
\int y^2 \mathbb{1}_{\{ y\leq x \}}r(n) \pi^{+}(ndy) 
 \xrightarrow[n \rightarrow \infty] {} c_\alpha \int  y^2 \mathbb{1}_{\{ y\leq x \}} \frac{dy}{y^{1+\alpha}}.
\end{equation*}
Using an argument of monote class to deduce the above convergence over intervals $I\subset (0,\infty)$. Thus we obtain the convergence of the Laplace transform of the measure $\mu$. This complete the proof because $\mu$ is regularly varying at infinity with indice $2-\alpha$ and its Laplace transform satisfies the identity (\ref{intasint}).

\section{Proof of Proposition \ref{convfunsta} }

Before proving Proposition \ref{convfunsta}, we would like to present some basic aspect of functional convergence of stochastic process, further details can be found in \cite{Shiryaev}. Is well know that the law of a L\'evy process $\{ X_t : t\geq 0 \}$ on $\mathbb{R}^d$ is determined by that of random variable $X_1$, which is infinitely divisible random variable, and according to the L\'evy-Khintchine formula has characteristic exponent  
$$
\Psi( \mathbf{u} ) = i \mathbf{u} \cdot \mathbf{b} - \frac{1}{2} \mathbf{u}\cdot c\mathbf{u}^T + \int \left( e^{i \mathbf{u}\cdot \mathbf{x} } -1 -i \mathbf{u} \cdot \mathbf{h}(\mathbf{x}) \right)F(d\mathbf{x})
$$ 
where $\mathbf{b}\in\mathbb{R}^d$, $c$ is a $d\times d$ symmetric nonnegative matrix, F is a positive measure on $\mathbb{R}^d$ with $F(\{ \mathbf{0} \})=0$ and $\int(1\wedge |\mathbf{x}|^2) F(d\mathbf{x})<\infty$, $\mathbf{h}$ is a truncation function from $\mathbb{R}^d$ to $\mathbb{R}^d$, that is, bounded measurable satisfying 
$$ 
\mathbf{h}(\mathbf{x})= o(|\mathbf{x}|), \quad |\mathbf{x}|\rightarrow 0.
$$
Hence an infinitely divisible distribution, and therefore a L\'evy process, is uniquely characterized by the triple $(\mathbf{b}, c, F)$.
Another useful related equality is a $d\times d$ symmetric nonnegative matrix, called the modified second characteristic, and defined as follows 
$$
\tilde{c}^{ij}= c^{ij} + \int h^i(\mathbf{x}) h^j (\mathbf{x}) F(d\mathbf{x}), \quad i,j=1,2,...,d. 
$$
According to Theorem VII.2.9 of \cite{Shiryaev}, if $\{F_n\}_{n\geq1}$ is a sequence of infinitely divisible distributions on $\mathbb{R}^d$. Then $F_n\rightarrow F$ weakly if and only if
\begin{align*}
\mathbf{b}&_n \rightarrow \mathbf{b}\\
\tilde{c}&_n \rightarrow \tilde{c}\\
F&_n(\mathbf{g})  \rightarrow F(\mathbf{g}) \quad \text{for all } \mathbf{g}\in C_1(\mathbb{R}^d),
\end{align*}
where $C_1(\mathbb{R}^d)$ is a convergence-determing class for the weak convergence induced by all continuous bounded non-negative functions $\mathbb{R}^d\rightarrow \mathbb{R}$, vanishing at the origin and with limit at infinity.  
We stress that here we consider characteristics relatives to a continuous truncation function, $\mathbf{h}$.

In a more general sense, a $d$-dimensional semimartingale $W$, has associated a characteristic triplet $(B,C,\nu)$ consisting in:
\begin{itemize}
\item [-] $B= (B^i)_{i\leq d}$ a predictable process with components of finite variation over each interval $[0,t]$.
\item [-] $C=(C^{ij})_{i,j\leq d}$ a continuous process, namely
$$
C^{ij} = \langle W^{i,c}, W^{j,c}  \rangle,
$$
where $W^c$ is the continuous martingale part of $W$.
\item [-] $\nu$ a predictable random measure on $\mathbb{R}_+\times\mathbb{R}^d$. 
\end{itemize}
A second modified characteristic $\tilde{C}$ is also defined, 
$$
\tilde{C}^{ij}_t = C^{ij}_t + (h^i h^j) * \nu_t - \sum_{s\leq t} \left( \int h^i (\mathbf{x})\nu( \{ s\} \times d\mathbf{x} ) \right) \left( \int h^j (\mathbf{w})\nu( \{ s\} \times d\mathbf{w} )   \right).
$$
If $W$ has no fixed times of discontinuity, in which case $B$ is continuous, and $|h(x)|^2* \nu_t<\infty$, it reduces to
$$
\tilde{C}^{ij}_t = C^{ij}_t + (h^i h^j) * \nu_t. 
$$
According to Theorem VII.3.4 of \cite{Shiryaev}, the necessary and sufficient conditions to assure the functional convergence of a sequence of semimartingales $W^n$ towards $W$ are given also in terms of their characteristics: \begin{align} \label{condconvfun}
&\sup_{s\leq t}  | B_s^n -B_s| \rightarrow 0,  \text{ for all } t\geq 0, \nonumber \\
&\tilde{C}^n  \rightarrow \tilde{C},  \text{ for all } t\in D, \nonumber \\
&\mathbf{g} *\nu^n_t \rightarrow \mathbf{g} *\nu_t , \text{ for all } t\in D, \mathbf{g}\in C_1(\mathbb{R}^d) ,
\end{align}
where $D$ is a dense subset of $\mathbb{R}_+$.

We now turn to prove the convergence claimed in Proposition \ref{convfunsta}. In order to apply the Theorem VII.3.4 of \cite{Shiryaev}, first we will prove the convergence of the characteristics of the process
\begin{equation} \label{keyproc}
\tilde{ \mathbf{S} }^n_{N_{r(n)t}} - (  r(n)t /n  ,0 ), \quad t\geq0,
\end{equation}
where 
$$
\tilde{ \mathbf{S} }^n_k =\sum_{i=1}^{k} \left( \xi_i^{(cn)} /n,  \xi_i^{(mn)} / r(n)p(n) \right) , \quad k\in\mathbb{N},
$$ 
and $\{N_t,t\geq 0\}$ is a Poisson process with parameter one, independent of the sequence $$\xi^{(n)} = \left\{ \left( \xi_k^{(cn)} ,  \xi_k^{(mn)} \right) : k\in \mathbb{Z}_+ \right\}.$$
The following Lemma establishes the previous statement. We will use this result as a device to study the characteristics of $\bar{ \mathbf{S} }^n_{ \lfloor r(n)t \rfloor} $, which are closely related to those of $\tilde{ \mathbf{S} }^n_{N_{r(n)t}} - (  r(n)t /n  ,0 )$.

\begin{lemma}\label{solVic}
The process defined in (\ref{keyproc}) is a semimartingale with characteristics relatives to a continuous truncation function $\mathbf{h}$ given by
\begin{align} \label{chawrp}
&\mathbf{b}^n_t   =  r(n)t \mathbb{E}\left[   \mathbf{h}\left(  b(n) \xi^{(n)}  \right)   \right] - (  r(n)t /n  ,0 ),  \nonumber \\
&c_t^{n,ij}   =0,  \quad \tilde{c}_t^{n,ij}=  r(n)t \mathbb{E}\left[ h_i \left( b(n)\xi^{(n)}  \right)  h_j \left(  b(n)\xi^{(n)}   \right) \right], \quad i,j=1,2, \nonumber \\
&F^n_t  (d\mathbf{x}) =  r(n)t \pi(d\mathbf{x}),
\end{align}
where $b(n)\xi^{(n)} = \left( \xi^{(cn)}/n ,  \xi^{(mn)}/r(n)p(n) \right)$, $ \pi^{(n)}(d\mathbf{x})=\mathbb{P}( \xi^{(cn)}\in dx_1, \xi^{(mn)}\in dx_2 )$.
Moreover, in the regime (\ref{hipcolapit}) and (\ref{hipnew}) we have the following weak convergence in the sense of Skorohod topology  
\begin{equation}\label{convppc}
\left( \left(  \tilde{ \mathbf{S} }^n_{N_{r(n)t}} - (  r(n)t /n  ,0 )  : t\geq0 \right), \mathbb{P}_{a(n)}^{p(n)} \right) \Longrightarrow \left(X_t,t : t\geq0 \right),
\end{equation}
where $X_t$ is a spectrally positive $\alpha$-stable process with parameter $\alpha\in(1,2)$. In particular, we obtain the convergence of the characteristics in (\ref{chawrp}) towards those relatives to $\left((X_t,t): t\geq0\right)$ and characteristic exponent $c_\alpha |\lambda|^\alpha$, that is 
\begin{align} \label{chasta}
& \mathbf{b}_t = \left(t \left( \int_{(0,\infty)} \lambda (h(y)-y)  c_\alpha y^{-(\alpha+1)}dy\right),t \right), \nonumber \\
& c_t^{ij} =0,  \quad  \tilde{c}_t^{ij}= \mathbb{E}\left[ h_i \left( X_t  \right)  h_j \left( X_t  \right) \right]  \quad i,j=1,2, \nonumber \\
& F_t(d\mathbf{x}) = t c_\alpha x_1^{-(\alpha+1)}dx_1\delta_0(dx_2),
\end{align}
where $h$ is a continuous truncation function from $\mathbb{R}$ to $\mathbb{R}$. 
\end{lemma}
\begin{proof}[Proof of Lemma \ref{solVic}]
Note that for $\mathbf{u}=(\lambda,\theta)\in\mathbb{R}^2$,
\begin{equation}\label{expchappc} 
\mathbb{E}\left( e^{ \text{i} \mathbf{u} \cdot \tilde{ \mathbf{S} }^n_{N_{ r(n)t }} } \right) 
= e^{  r(n)t   \left(  \psi _n\left(  \frac{ \lambda } { n } , \frac{ \theta } { r(n)p(n) }  \right) -1 \right) } , \quad t\geq 0.
\end{equation}
Then the exponent in the righthand side of the previous equality can be written as follows
\begin{equation} \label{expsum0}
t \left[ \text{i} \int_{ (0,\infty) } \int_{ (0,\infty) }  \mathbf{u}  \cdot \mathbf{h}( b(n) \mathbf{x})r(n)  \pi(d\mathbf{x}) +  \int_{ (0,\infty) } \int_{ (0,\infty) } \left(  e^{ \text{i} \mathbf{u}  \cdot b(n) \mathbf{x} } -1 - \text{i} \mathbf{u} \cdot \mathbf{h}(b(n) \mathbf{x})  \right)  r(n)  \pi^{(n)}(d\mathbf{x}) \right], 
\end{equation} 
where $b(n) \mathbf{x}= (x_1/n,x_2/r(n)p(n))$, $\pi^{(n)}(d\mathbf{x}) = \mathbb{P}( \xi^{(cn)}\in dx_1, \xi^{(mn)}\in dx_2 )$. From here $\tilde{\mathbf{S}}_{N_{ r(n)t }}$ is infinitely divisible, also we can deduce that the characteristics of the process $\{\tilde{ \mathbf{S} }^n_{N_{r(n)t}} - (  r(n)t /n  ,0 ) : t\geq 0\}$ are given by (\ref{chawrp}). Thanks to Theorem II.3.11 of \cite{Shiryaev} this process is a L\'evy process and a semimartingale.

Besides, to get the convergence in (\ref{convppc}) we shall prove the convergence of the characteristic functions. This fact is verified using Corollary \ref{corTLlrlclm}, together with the fact that conditionally to $\xi^{(+)}=k$ the distribution of $\xi^{(cn)}$ is Binomial with parameter $(k,1-p(n))$, as well as the assumption that $\xi^+$ has mean $1$.
\begin{align*}
t &\int_{ (0,\infty) }  \left( e^{ - ( (1-p(n) )( 1- e^{\text{i} \lambda/n } ) - p(n)(1-e^{\text{i} \theta / r(n)p(n) } ) ) y} -1 - \text{i}\mathbf{u} \cdot \mathbf{h} \left(  \frac{y}{n}(1-p(n)) ,\frac{y}{r(n)p(n)}p(n)  \right) \right) r(n) \pi^+(dy) \\
& + t\text{i} \int_{ (0,\infty) } \left( \mathbf{u} \cdot \mathbf{h} \left(  \frac{y}{n}(1-p(n)), \frac{y}{r(n)p(n)}p(n) \right) -\left( \frac{\lambda}{n}(1-p(n)) +  \frac{\theta}{r(n)p(n)} p(n) \right)y \right)r(n) \pi^+(dy) \\
& + \text{i} \left( \frac{\lambda}{n}(1-p(n)) +  \frac{\theta}{r(n)p(n)} p(n)  \right) r(n)t.
\end{align*}
Then by the continuity of $\mathbf{h}$ and the assumptions (\ref{hip}) and (\ref{rpro}), the previous display becomes 
\begin{align*}
\thicksim t\int_{ (0,\infty) } & \left( e^{ - ( (1-p(n) )( 1- e^{\text{i} \lambda/n } ) - p(n)(1-e^{\text{i} \theta / r(n)p(n) } ) ) y} -1 - \text{i} \lambda h \left(  \frac{y}{n} \right) \right) r(n)\pi^+(dy) \\
& \quad \quad  \quad  +t \text{i} \int_{ (0,\infty) } \left( \lambda h \left(  \frac{y}{n} \right) - \left( \frac{\lambda}{n}(1-p(n)) +  \frac{\theta}{r(n)p(n)} p(n) \right)y \right) r(n)\pi^+(dy) \\
&  \quad \quad  \quad + \text{i}  \left( \frac{\lambda}{n}(1-p(n)) +  \frac{\theta}{r(n)p(n)} p(n) \right) r(n)t.
\end{align*}
where $h$ is the truncation function from $\mathbb{R}$ to $\mathbb{R}$ obtained as projection of $\mathbf{h}$ in the second coordinate, that is,
\begin{align}\label{aprofin}
\thicksim t\int_{ (0,\infty) } & \left( e^{ - ( (1-p(n) )( 1- e^{\text{i} \lambda/n } ) - p(n)(1-e^{\text{i} \theta / r(n)p(n) } ) ) ny} -1 - \text{i}\lambda h (  y  ) \right)r(n) \pi^+(ndy) \nonumber \\
& + t\text{i} \int_{ (0,\infty) } \left( \lambda h( y )- \left(\lambda(1-p(n)) +  \frac{\theta}{r(n)p(n)} n p(n) \right)y \right)r(n) \pi^+(ndy)  \nonumber \\
& + \text{i}  \left( \frac{\lambda}{n}(1-p(n)) +  \frac{\theta}{r(n)p(n)} p(n) \right) r(n)t.
\end{align}
Finally we have the convergence  
\begin{equation}\label{autreconv}
\mathbb{E}\left( e^{ \text{i}  \mathbf{u} \cdot (\tilde{ \mathbf{S} }^n_{N_{ r(n)t }} -  (r(n) t /n,0)  ) } \right)\xrightarrow [n\rightarrow\infty]{} 
e^{t( \int_{ (0,\infty) } \left( e^{\text{i} \lambda y} -1- \text{i} \lambda y  \right)c_\alpha y^{-(\alpha+1)}dy) +\text{i} t \theta},
\end{equation}
where $c_\alpha$ is a constant depending on $\alpha$ that appears in Lemma  \ref{convmedidas}. Indeed, this is a consequence of the arguments in this  Lemma together with the hypothesis (\ref{hip})
\begin{align*} 
\int_{ (0,\infty) } & \left( e^{ - ( (1-p(n) )( 1- e^{\text{i} \lambda/n } ) - p(n)(1-e^{\text{i} \theta / r(n)p(n) } ) ) ny} -1 - \text{i}  \lambda h(  y ) \right) r(n)\pi^+(ndy) \\
& \xrightarrow [n\rightarrow\infty]{} t \int_{ (0,\infty) } \left( e^{\text{i} \lambda y} -1- \text{i}  \lambda h( y )  \right)c_\alpha y^{-(\alpha+1)}dy,
\end{align*}
while the second adding converges to $$\text{i} t \int_{(0,\infty)} \lambda (h(y)-y)  c_\alpha y^{-(\alpha+1)}dy. $$ To end we observe that $r(n)\in RV_{\infty}^\alpha$ with $\alpha\in(1,2]$ and $p(n)\thicksim cn^{-1}$ implies that
$$
\frac{r(n) p(n)}{n } \xrightarrow [n\rightarrow\infty]{} 0.
$$
Then 
$$
\text{i}  \left( \frac{\lambda}{n}(1-p(n)) +  \frac{\theta}{r(n)p(n)} p(n) \right) r(n)t - \text{i}  \frac{\lambda }{n}r(n)t,
$$
converges to $\text{i}\theta t$. From (\ref{autreconv}) the characteristic of $(X_t,t)$ are given by (\ref{chasta}). As a consequence of Theorem VII.2.9 of \cite{Shiryaev} we have the convergence of the characteristics. Finally the characteristic exponent of $X_t$ is $c_\alpha \lambda^\alpha$.\end{proof} 

We have now all the elements to prove Proposition \ref{convfunsta}.

\begin{proof}[Proof of Proposition \ref{convfunsta}]
Thanks to Theorem 2.3.11 of \cite{Shiryaev}, $\tilde{\mathbf{S}}_{\lfloor r(n)t \rfloor}^{(n)}$ is a semimartingale with characteristics relatives to $\mathbf{h}:\mathbb{R}^2\rightarrow \mathbb{R}^2$, given by
\begin{align*}
&B_t^n  = \lfloor r(n)t \rfloor
\mathbb{E}\left[ \mathbf{h}\left(  b(n)\xi^{(n)}  \right)   \right]  - (\lfloor r(n)t \rfloor /n,0) \\
&C_t^{n,ij}  =0, \quad \tilde{C}_t^{n,ij}= \lfloor r(n)t \rfloor \left( \mathbb{E} \left[  h_i\left(  b(n)\xi^{(n)}  \right) h_j\left(  b(n)\xi^{(n)} \right) \right]- \mathbb{E}\left[ h_i\left(b(n)\xi^{(n)} \right) \right] \mathbb{E}\left[h_j( b(n)\xi^{(n)} ) \right] \right) \quad i,j=1,2.  \\
&\mathbf{g} *\nu^n_t=  \lfloor r(n)t \rfloor \mathbb{E}\left( \mathbf{g}\left( b(n)\xi^{(n)}   \right)  \right) \quad \text{for } \mathbf{g}\geq 0 \text{ Borel},
\end{align*}
Now only remains to verify the conditions (\ref{condconvfun}) where the limit characteristics agree with these in (\ref{chasta}). In this direction we recall that $\mathbf{b}^n_t   =  r(n)t \mathbb{E}\left[   \mathbf{h}\left(  b(n) \xi^{(n)}  \right)   \right] - (  r(n)t /n  ,0 )$ and observe
$$
|B_s^n -\mathbf{b}_s |  \leq \left| \lfloor  r(n) s\rfloor-r(n)s \right| \mathbb{E}\left[ \mathbf{h}\left(  b(n)\xi^{(n)}  \right)   \right] + | ( r(n)s/n,0 ) - (\lfloor r(n)s \rfloor/n,0)  | +\left| \mathbf{b}^n_s  -\mathbf{b}_s  \right|.
$$
Then using the properties of the floor function we obtain
$$
|B_s^n -\mathbf{b}_s|\leq \mathbb{E}\left[ \mathbf{h}\left(  b(n)\xi^{(n)}  \right)   \right]
+  | (s/n,0) |  + \left| \mathbf{b}^n_s  -\mathbf{b}_s  \right|.
$$
Thus by the convergence of $\mathbf{b}_t^n$ established in the previous lemma, together with $r(n)\to\infty$
$$
\sup_{s\leq t}|B_s^n -\mathbf{b}_s|
\leq 
\mathbb{E}\left[ \mathbf{h}\left(  b(n)\xi^{(n)}  \right)  \right] +  | (t/n,0) |  +  \left| \mathbf{b}^n_t  -\mathbf{b}_t  \right|
\xrightarrow[n\to\infty]{}0,
$$
hence we have the first condition in (\ref{condconvfun}). In order to determine the second one, let  $b_t^{n,i}$ be the $i$-th coordinate of $\mathbf{b}^n_t$, $i=1,2$. Once more, applying the properties of the floor function we have 

\begin{align*}
 \left(1-\frac{1}{r(n)t}\right)\tilde{c}_t^{n,ij}  - \frac{1}{r(n)t} b_t^{n,i}b_t^{n,j} - & \frac{r(n)t}{n} \mathbb{E}\left[ h_2\left(  b(n)\xi^{(n)}  \right)  \right] \\
 & \leq \tilde{C}_t^{n,ij}  \leq \tilde{c}_t^{n,ij} + \frac{ 1- r(n)t }{ (r(n)t)^2  } b_t^{n,i}b_t^{n,j}  +\frac{1-r(n)t}{n} \mathbb{E}\left[ h_2\left(  b(n)\xi^{(n)}  \right)  \right]. 
\end{align*}
Also
$$
\left(1-\frac{1}{r(n)t}\right)\tilde{c}_t^{n,22} - \frac{ 1 }{ r(n)t  } (b_t^{n,2})^2  \leq \tilde{C}_t^{n,22}\leq \tilde{c}_t^{n,22} +\frac{ 1- r(n)t }{ (r(n)t)^2  }  (b_t^{n,2})^2,
$$
and
$$
\left(1-\frac{1}{r(n)t}\right) \tilde{c}_t^{n,11} - \frac{1}{ r(n)t }  (b_t^{n,1})^2   \leq \tilde{C}_t^{n,11}\leq \tilde{c}_t^{n,11} + \frac{ 1- r(n)t }{ (r(n)t)^2  } (b_t^{n,1})^2 +  b_t^{n,1} + \frac{1}{n^2}[ r(n)-1].
$$
As consequence of the convergence $\mathbf{b}^n_t\rightarrow \mathbf{b}_t$, $b^{n,i}_t$ converges for $i=1,2$. Then the above inequalities impliy $\tilde{C}_t^{n,ij} \rightarrow \tilde{c}_t^{ij}$, because $r(n)\in RV_{\infty}^\alpha$ and $\alpha\in(1,2)$. It is easily prooved that also $\mathbf{g} *\nu^n_t \rightarrow \mathbf{g} * F_t$ for all $g$ using that  $F^n_t=  r(n)t \pi(d\mathbf{x})$ converges to $F_t$, as we proved in the previous lemma, together with properties of the floor function.
\end{proof}


\bibliographystyle{abbrv}
\bibliography{bibliomath.bib}

\end{document}